\documentclass[10pt]{article}

\usepackage{amsmath,latexsym,amssymb}

\usepackage{esint}
\usepackage{graphicx}
\usepackage{bm}
\usepackage{comment}
\usepackage{epsfig}

\usepackage{mathrsfs}
\usepackage{enumerate}
\usepackage{color}
\usepackage{amsfonts}
\usepackage{verbatim}
\usepackage{amsthm}
\usepackage{dsfont}
\usepackage{url}
\usepackage{esint}
\usepackage{tocloft}
\usepackage{tikz}
 \usetikzlibrary{arrows}    
 \usetikzlibrary{calc}  
\usetikzlibrary{decorations.markings}
\usepackage[mathcal]{euscript}

\DeclareMathOperator{\inte}{int}

\DeclareMathOperator*{\wsc}{\overset{*}{\rightarrow}}

\theoremstyle{definition}
\newtheorem{definition}{Definition}[section]
\newtheorem{assumption}[definition]{Assumption}
\theoremstyle{remark}
\newtheorem{remark}[definition]{Remark}

\theoremstyle{plain}
\newtheorem{theorem}[definition]{Theorem}
\newtheorem{lemma}[definition]{Lemma}
\newtheorem{proposition}[definition]{Proposition}

\numberwithin{equation}{section}

\usepackage{authblk}
\title{Asymptotics of polynomials orthogonal over circular multiply connected domains}
\author[1]{James Henegan \thanks{Email: jahenegan@umc.edu}}
\author[2]{Erwin Mi\~{n}a-D\'{\i}az \thanks{Corresponding author; Email: minadiaz@olemiss.edu}}
\affil[1]{Department of Data Science\\
University of Mississippi Medical Center,
2500 North State Street,
Jackson, MS 39216}
\affil[2]{The University of Mississippi,
Department of Mathematics,
Hume Hall 305,P.~O.~Box 1848,
University, MS 38677-1848, USA.}

\date{\today}                     
\newcommand{\caliD}{\mathcal{D}}
\newcommand{\caliP}{\mathcal{P}}

\newcommand{\caliT}{\mathcal{T}}

\newcommand{\caliF}{\mathcal{F}}
\newcommand{\cj}{\overline}
\newcommand{\T}{\mathbb{T}}
\newcommand{\N}{\mathbb{N}}
\newcommand{\C}{\mathbb{C}}
\newcommand{\Z}{\mathbb{Z}}
\newcommand{\R}{\mathbb{R}}

\newcommand{\wt}{\widetilde}

\providecommand{\keywords}[1]
{
  \small	
  \textbf{\textit{Keywords---}} #1
}

\begin{document}

\maketitle

\begin{abstract}
  Let $\caliD$ be a domain obtained by removing, out of the unit disk $\{z:|z|<1\}$, finitely many mutually disjoint closed disks, and for each integer $n\geq 0$, let $P_n(z)=z^n+\cdots$ be the monic $n$th-degree polynomial satisfying the planar orthogonality condition $\int_{\caliD}P_n(z)\cj{z^m}dxdy=0$, $0\leq m<n$.  Under a certain assumption on the domain $\caliD$, we establish  asymptotic expansions and formulae that describe the behavior of $P_n(z)$ as $n\to\infty$ at every point $z$ of the complex plane. We also give an asymptotic expansion for the  squared norm  $\int_\caliD|P_n|^2dxdy$. 
 \end{abstract}
 

\section{Introduction and main results}

Let $\caliD$ be a bounded domain in the complex plane $\C$, and for each integer $n\geq 0$, let $P_n(z)=z^n+\cdots$ be the unique monic polynomial of degree $n$ satisfying the orthogonality condition 
\begin{align}\label{ortho-Pn}
\int_{\caliD}P_n(z)\cj{z^m}dA(z)=0,\quad 0\leq m\leq n-1,
\end{align}
where $A$ is the two-dimensional Lebesgue measure divided by $\pi$: $dA=\pi^{-1}dxdy$.

With
\begin{align}\label{defkappa_n}
 \kappa_n:=\left(\int_\caliD|P_n|^2dA\right)^{-1/2},\quad n=0,1,\ldots,
\end{align}
the polynomials $p_n:=\kappa_n P_n$ form an orthonormal sequence:
\[
\int_\caliD p_n\cj{p_m}dA=\delta_{n,m}, \quad n, m\geq 0.
\]
To indicate their dependence on $\caliD$, we will write $P_n(z,\caliD)$, $p_n(z,\caliD)$, and $\kappa_n(\caliD)$.

 We note that the leading coefficient $\kappa_n$ carries the important extremal property  
\begin{align}\label{extremality}
 \kappa_n^{-2}=\min\int_\caliD|P|^2dA,
\end{align}
the minimum being taken over all monic polynomials $P$ of degree $n$.

In this paper, we establish formulae that describe the behavior of $P_n$ and $\kappa_n$, as $n\to\infty$, for a multiply connected domain $\caliD$ whose boundary consists of finitely many mutually disjoint circles. Such a domain is commonly referred to as a circular multiply connected domain, briefly, a CMCD.

For disks, circles, and exterior of circles, we will use the notation 
\begin{align*}
D(c,r):= {} &\{z \in \C : |z - c | <r\},\\
\T(c,r):={} & \{z \in \C : |z - c | =r\},\\
\Delta(c,r):={} & \{z \in \cj{\C} : |z - c | >r\}.
\end{align*}
Also, for every integer $s\geq 1$, we let 
\[
 \N_s:=\{1,2,\ldots,s\}.
\]

After a translation and a scaling, we can always make a CMCD have the form  
\begin{equation}\label{notation:CMCD}
\mathcal{D}= D(0,1) \setminus  \bigcup_{j=1} ^s \cj{D(c_j, r_j)}\ ,
\end{equation}
where $s \geq 1$ and 
\[
\cj{D(c_j,r_j)}\subset D(0,1),\quad \cj{D(c_j,r_j)}\cap\cj{D(c_k,r_k)}=\emptyset, \quad  1\leq j\not=k\leq s. 
\]

It is known that if $\caliD$ is any subdomain (or more generally, a measurable subset) of  $D(0,1)$  containing an annulus of the form $A_r=\{z:r<|z|<1\}$ for some $r\in[0,1)$, then 
\begin{align}\label{leading-asympt}
\lim_{n\to\infty}\frac{\kappa_n(\caliD)}{\sqrt{n+1}}=1
\end{align}
and
\begin{align}\label{exterior-asympt}
\quad \lim_{n\to\infty}\frac{P_n(z,\caliD)}{z^n}=1, \quad z\in \Delta(0,r).
\end{align}

This result is an instance of a more general one by Korovkin \cite[formulae (10), (14), (15)]{korov} for polynomials orthogonal with weights over domains bounded by analytic Jordan curves. The limit in  \eqref{exterior-asympt} takes place uniformly on closed subsets of $\Delta(0,r)$, and geometric estimates for the speed of convergence in \eqref{leading-asympt} and \eqref{exterior-asympt}  that depend on the number $r$ are also given in \cite{korov}. Thus, for a  $\caliD$ as in \eqref{notation:CMCD}, we know that \eqref{exterior-asympt} holds true with 
\begin{align}\label{rhox}
r=1/\rho_x:=\max_{j\in \N_s}(|c_j|+r_j).
\end{align}

With the exception of \cite[Proposition 3.2]{SS} (see the discussion toward the end of Subsection \ref{asymptoticresults}), no further exploration on the asymptotic properties of polynomials orthogonal over a multiply connected domain seems to have been pursued in the existing literature. We will be able to expand on \eqref{exterior-asympt} by giving a series representation for $P_n$ (Theorem \ref{thm-poly}) that yields, after further analysis, the asymptotic behavior of $P_n(z)$ at every point  $z$ of the complex plane. In particular, we prove that the limit in \eqref{exterior-asympt} extends to a maximal domain of the form $\Delta(0,\rho_a)$, where $\rho_a$ is a number that is determined by the inner circles bounding $\caliD$ and is such that $\rho_a<\rho_x^{-1}$. We  give the exact rate of convergence in \eqref{exterior-asympt}, which differs according to whether $|z|>\rho_x$ or $\rho_a< |z|\leq \rho_x$, with $\rho_x$ as in \eqref{rhox}.  Indeed, from the series representation we can get a full asymptotic description of the error term, which particularly 
for $|z|>\rho_x$, turns into a nice asymptotic expansion for $P_n$.  We will also be able to refine \eqref{leading-asympt} by giving a full asymptotic expansion for $\kappa_n^{-2}$.

Our method of proof requires an assumption on $\caliD$ that we have proven to hold true in many cases. We expect this assumption to be, indeed, satisfied by every CMCD.

It is possible to extend some of the results of this paper to domains of the form $\varphi(\caliD)$, where $\caliD$ is as in \eqref{notation:CMCD} and  $\varphi$ is a conformal map of $D(0,1)$ onto the interior of an analytic Jordan curve, see \cite{Henegan} for details.

The asymptotic properties of orthogonal polynomials over planar regions have been the focus of attention of past and many recent works. When the domain of orthogonality is bounded by a Jordan curve  with some degree of smoothness (analytic, piecewise analytic, H\"{o}lder continuous, quasiconformal), strong asymptotics and/or zero distribution results have been derived in \cite{BS,Carl,DM1,DM2,DMN,LSS,MS,M1,SS2,S,Suetin}, and for orthogonality with weights, in \cite{korov,M2,MSS}. Logarithmic/zero asymptotics with applications to shape reconstruction have been given in \cite{GPSS,SSSV} for polynomials orthogonal over an archipelago (a finite union of Jordan domains). The papers \cite{MinaSimanek,SS,Simanek1,Simanek2}, although more general in scope, also carry important implications for planar orthogonality.

\subsection{Preliminaries}

Let $\caliD$ be given by \eqref{notation:CMCD}. For each $j \in \N_s$, there exists a unique pair of numbers $a_j \in D(0,1)$ and $\sigma_j \in (0,1)$ such that the M\"{o}bius transformation 
 \[
\Phi_j(z) := \frac{\cj{a}_j}{|a_j|}\frac{z-a_j}{1-\overline{a}_jz}
\]
maps $D(c_j,r_j)$ onto  $D(0,\sigma_j)$:
\[
\Phi_j(D(c_j,r_j))=D(0,\sigma_j).
\]  

It is easy to verify the relations 
\begin{align}\label{formulas-for-c-and-r}
 c_j=\frac{a_j(1-\sigma_j^2)}{1-|a_j|^2\sigma_j^2},\quad r_j=\frac{\sigma_j(1-|a_j|^2)}{1-|a_j|^2\sigma_j^2}, \quad j\in \N_s,
\end{align}
whence we get that for all $j\in \N_s$,
\begin{align}\label{formulas-for-c-and-r-1}
 |c_j|\leq |a_j|\quad \text{and}\quad  r_j\leq \sigma_j, 
\end{align}
equality holding in each case  if and only if $c_j=0$.

The function $\Phi_j$ is an automorphism of the unit disk whose inverse $\Phi_j^{-1}$ is given by  
\[
\Phi_j^{-1}(t) =\frac{a_j}{|a_j|} \frac{t+|a_j|}{1+|a_j|t}.
\]

For each $j \in \N_s$, we  define 
    \begin{align}\label{def-Tj}
    T_j (z) := \Phi_j^{-1} (\sigma_j ^2 \Phi_j  (z)), \quad z \in \cj{\C},
    \end{align}
and associate to $\mathcal{D}$ the family $ \caliT^*$ of all finite compositions of transformations $T_j$:
         \begin{align}\label{def-T*-family}
   \caliT^*:= \{T_{j_n} T_{j_{n-1}} \cdots T_{j_2} T_{j_1}: n \in \N, \ j_k \in \N_s, \ k\in  \N_n\}.
    \end{align} 
We adjoin the identity map $T_0(z) \equiv z$
to $ \caliT^*$ to form 
\begin{align}\label{def-T-family}     
\mathcal{T}:= \caliT^*\cup \{T_0 \}.
\end{align}
    
As illustrated in Figure \ref{t-sub-alpha}, let 
\[
\rho_a:=\max_{j\in \N_s}|a_j|.
\]

The disk $D(0,\rho_a^{-1})$ is the largest disk about the origin in which every $\tau \in \mathcal{T}$ is analytic. The validity of our asymptotic results rests upon the following assumption.  
    
\begin{assumption}\label{assumption}
   The series $\sum_{\tau \in \mathcal{T}} |\tau'(z)|$ converges locally uniformly on \\$D(0,\rho_a^{-1})$.
\end{assumption}
 Assumption \ref{assumption} is equivalent to the convergence of $\sum_{\tau \in \mathcal{T}} |\tau'(z)|$ for some  $z\in D(0,\rho_a^{-1})$ (see Section \ref{section3.1}, in particular, the inequalities \eqref{series-sandwich}). We will  establish its validity in a number of cases that we 
 summarize in the following proposition.

\begin{figure}[t]

\begin{center}
\begin{tikzpicture}[scale=2]


\draw[densely dashed] (0,0) circle (1cm);

\filldraw
(0,0) circle (.75pt) node[align=center,   below] {$  0$};

\filldraw
(1,0) circle (.75pt) node[align=center,   below] {$  1$};

\filldraw
(0.649519,.375) circle (.75pt) node[align=center, left] {$a_1$};

\filldraw
(-0.467617,0.586374) circle (.75pt) node[align=center, left] {$a_2$};

\filldraw
(-0.367617,-0.286374) circle (.75pt) node[align=center, left] {$a_3$};

\draw[densely dashed] (0,0) circle (0.75cm);

\fill[gray!50!white] (-3+0,0) circle (1cm);

\draw[densely dashed] (-3+0,0) circle (1cm);

\fill[white!20!white] (-3+0.521244,0.30094) circle (0.329154cm);

\draw[densely dashed] (-3+0.521244,0.30094) circle (0.329154cm);
\filldraw
(-3+0.649519,.375) circle (.75pt) node[align=center, left] {$a_1$};

\fill[white!20!white] (-3-0.431646,0.541268 ) circle (0.192308cm);
\draw[densely dashed] (-3-0.431646,0.541268 ) circle (0.192308cm);
\filldraw
(-3-0.467617,0.586374) circle (.75pt) node[align=center, below] {   $a_2$};

\fill[white!20!white] (-3-0.255227,-0.198822) circle (0.509542);
\draw[densely dashed] (-3+-0.255227,-0.198822) circle (0.509542);

\filldraw
(-3-0.367617,-0.286374) circle (.75pt) node[align=center, left] {$a_3$};

\filldraw
(-3+0,0) circle (.75pt) node[align=center,   below] {$  0$};

\filldraw
(-3+1,0) circle (.75pt) node[align=center,   below] {$  1$};

\end{tikzpicture}

\caption{A CMCD with $\rho_a = |a_1|= |a_2| > |a_3|$. }

\label{t-sub-alpha}

\end{center}

\end{figure}
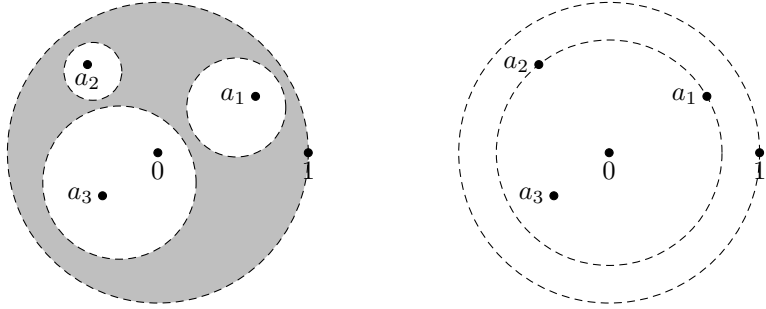

\begin{proposition}\label{thm-converge-cases} Let $\mathcal{D}$ be a CMCD as described by \eqref{notation:CMCD}. Assumption \ref{assumption} holds true whenever $\caliD$  satisfies any one of the following conditions:
\begin{itemize}
\item[i)]  $c_j \in (-1,1)$ for each $j \in \N_s$.
\item[ii)]
\begin{align}\label{ineq-radii-condition}
\sum_{j=1}^s\frac{r_{j}^2}{(1-|c_{j}|\rho_a)^{2}}<1.
\end{align}
\item[iii)]  $\mathcal{D} = \Psi (\tilde{\mathcal{D}})$, where $\tilde{\mathcal{D}}$ is a CMCD  that satisfies Assumption \ref{assumption} and $\Psi$ is an automorphism of the unit disk. In particular, 
this is the case for any $\mathcal{D}$ with one or two removed disks ($s=1,2$).
\end{itemize}
\end{proposition}

Note that since $\rho_a<1$, condition \eqref{ineq-radii-condition} is satisfied if 
\[
\max_{j\in\N_s}|r_j|^2<\frac{1}{\sum_{j=1}^s(1-|c_j|)^{-2}}.
\]
Hence, for any configuration of centers $c_j$, there is $r>0$ such that Assumption \ref{assumption} is verified by every $\caliD$ with $\max_{j\in\N_s}r_j<r$. 

Note also that since $\pi\sum_{j=1}^s r_j^2<1$, 
\eqref{ineq-radii-condition} is satisfied by every CMCD with $\max_{j\in\N_s}|c_j|<1-\pi^{-1/2}$.

Based on the variety of cases covered by Proposition \ref{thm-converge-cases},  we  believe  that Assumption \ref{assumption} is indeed satisfied by every CMCD. 
 
\subsection{Asymptotics for the orthogonal polynomials}  \label{asymptoticresults}
 
Let $\mathcal{D}$ be as described by \eqref{notation:CMCD}, and for each integer $n\geq 0$, let $P_n$ and $\kappa_n$ be defined by \eqref{ortho-Pn} and \eqref{defkappa_n}. The simplest situation is when $\caliD$ is bounded by two concentric circles, that is, when
\begin{equation}\label{annulus}
\mathcal{D} = D(0,1) \setminus \cj{D(0, r_1)}, \quad 0<r_1<1.
\end{equation}
This case is actually trivial, since for such a $\caliD$ we have
\begin{align}\label{PnSingleDisk}
P_n(z)=z^n, \quad \kappa_n=\sqrt{\frac{n+1}{1-r_1^{2n+2}}},\quad n\geq 0.
\end{align}

Thus, from now on  we will always assume  $\caliD$ is  such that for some $j\in \N_s$, $c_j\not=0$, which, by \eqref{formulas-for-c-and-r-1}, is equivalent to assume  that 
$\caliD$ is such that 
\[
\rho_a>0.
\]

For every $j\in \N_s$, let us define 
\begin{align}\label{def-of-xj-yj}
x_j:=\frac{c_j/|c_j|}{|c_j|+r_j},\quad y_j:=\frac{c_j/|c_j|}{|c_j|-r_j},
\end{align}
with the understanding that $y_j=\infty$ when $|c_j|=r_j$.  We can also express these quantities in terms of $a_j$ and $\sigma_j$ by the formulae
\begin{align*}
x_j=\frac{a_j}{|a_j|}\frac{1+|a_j|\sigma_j}{\sigma_j+|a_j|},\quad y_j=-\frac{a_j}{|a_j|}\frac{1-|a_j|\sigma_j}{\sigma_j-|a_j|}.
\end{align*}

These numbers obey the inequalities
\[
1<|x_j|<|a_j|^{-1},\qquad |x_j|<|y_j|,
\]
so that  
\begin{align*}
1<\rho_x:=\min_{j\in \N_s}|x_j|<\rho_a^{-1}.
\end{align*}

Let 
\begin{align}\label{e_j-def}
\epsilon_j:=\begin{cases}
1, & r_j\geq |c_j|,\\
-1, & r_j< |c_j|.
\end{cases}
\end{align}
Geometrically, $\epsilon_j$ distinguishes whether $0\in \cj{D(c_j,r_j)}$ (case $r_j\geq |c_j|$) or not. 

To every $j\in \N_s$, we associate a function $R_j(w,z)$ as follows. If $r_j\not=|c_j|$ (geometrically, if $0\not\in \T(c_j,r_j)$), we set
\begin{align}\label{def-Rjk}
\begin{split}
R_j(w,z):={} & \frac{\epsilon_j\left(y_j/x_j-1\right)}{\sqrt{(y_j/x_j)^2-1-w}}\times\frac{1-\frac{z}{x_j}+w\left(\frac{y_j-2x_j}{y_j-x_j}+\frac{y_j^2+x_j^2}{y_j^2-x_j^2}\frac{z}{x_j}\right)-\frac{x_jw^2}{y_j-x_j}}{(1-z/x_j)^2+w\left(1-\frac{2z}{y_j+x_j}\right) },
\end{split}
\end{align}
whereas if  $r_j=|c_j|$, we let
\begin{align}\label{def-Rjk-2}
\begin{split}
R_j(w,z):=-\frac{1-z/x_j+w\left(1+z/x_j\right)}{(1-z/x_j)^2+w}.
\end{split}
\end{align}

For each $z\not=x_j$, the function $R_j(w,z)$  is analytic (in the variable $w$) in a neighborhood of the origin, and its Maclaurin series
\begin{align}\label{coefficients-Rjk}
R_j(w,z)=\sum_{k=0}^\infty R_{j,k}(z)w^k
\end{align}
is easy to compute, albeit when $r_j\not=|c_j|$, the expressions that explicitly represent  the coefficients $R_{j,k}(z)$ quickly become cumbersome as $k$ grows. Still, one can easily see that  each $R_{j,k}$ is a rational function whose only pole is $x_j$, and 
\begin{align*}
 R_{j,k}(0)=\begin{cases}\frac{\epsilon_j(-1)^{k}\binom{1/2}{k}\left(2k\frac{y_j}{x_j}+1\right)\left(\frac{y_j}{x_j}-1\right)}{\left((y_j/x_j)^2-1\right)^{k+1/2}},& \ r_j\not=|c_j|,\\
 -1,&\ r_j=|c_j|,\ k=0,\\
 0, &\ r_j=|c_j|,\ k>0.
\end{cases}
 \end{align*}

Let 
\[
C_k(z):=\sum_{j:|x_j|=\rho_x}R_{j,k}(z),\qquad k\geq 0.
\] 

\begin{theorem}\label{thm-asymp-leading}For the leading coefficients $\kappa_n$, we have the asymptotic expansion (as $n\to\infty$)
\begin{align} \label{leading-coeff-asympt}
(n+1) \kappa_n ^{-2}\sim 1+\frac{\rho_x^{-2n-2}}{2\pi}\sum_{k=0}^\infty C_{k}(0)\frac{\Gamma(k+1/2)\Gamma(n-k+3/2)}{\Gamma(n+2)}.
\end{align}
\end{theorem}

\begin{remark} The expansion \eqref{leading-coeff-asympt} ``degenerates"  (i.e., $C_{k}(0)=0$ for all $k\geq 1$) if $r_j=|c_j|$ whenever $|x_j|=\rho_x$. However, there can only be one $j\in\N_s$ obeying $r_j=|c_j|$, since this condition is equivalent to $0\in \T(c_j,r_j)$, so that the degeneration happens exactly when one of the circles, say $\T(c_1,r_1)$, passes through the origin, and 
\[
\T(c_j,r_j)\subset D(0,|c_1|+r_1),\quad j=2,\ldots,s.
\]
When this is the case,  there exists $0<\beta<1$ such that 
\begin{align} \label{degenerate-case}
(n+1) \kappa_n ^{-2}= 1-\frac{\rho_x^{-2n-2}}{2\sqrt{\pi}} \frac{\Gamma(n+3/2)}{\Gamma(n+2)} +O(\beta^n)
\end{align} 
as $n\to\infty$.
\end{remark}

The quantity
\begin{equation}\label{defmandMr}
m(r):=\max\{|T_j(z)|:|z|=r,\, j\in \N_s\},\quad r\in[0,\rho_a^{-1}],
\end{equation} 
is used in our next theorem to provide the rate of decay of the error term. To better understand the estimate, we mention that  (see Lemma \ref{lemma1} below)
\[
 m(r)< r,\qquad r\in(\rho_a,\rho_a^{-1}),
\]
and that
\[
m(\rho_x)/\rho_x=\rho_x^{-2}=\min_{r\in(\rho_a,\rho_a^{-1})}m(r)/r.
\]

\begin{theorem}\label{thm-cmcds} (a) For every $r>\rho_a$, we have 
\begin{align}\label{exterior-asymptotics}
\frac{P_n(z)}{z^n}=1 + \begin{cases}O((m(r)/r)^n),& \rho_a<r\leq\rho_x,\\
O(n^{-1/2}(m(\rho_x)/\rho_x)^n), & r>\rho_x,
                        \end{cases}
\end{align}
uniformly in $z\in \T(0,r)$ as $n\to\infty$. Moreover,
the asymptotic expansion 
\begin{align}\label{expansion-exterior-asymptotics}
\frac{P_n(z)}{z^n} \sim {}  1+
                        \frac{\rho_x^{-2n-2}}{2\pi}\sum_{k=0}^\infty C_{k}(z)\frac{\Gamma(k+1/2)\Gamma(n-k+3/2)}{\Gamma(n+2)}
\end{align} 
holds true uniformly on closed subsets of $\Delta(0,\rho_x)$ as $n\to\infty$.

(b) For all $n$ sufficiently large, 
\begin{equation}\label{cor-eqn-one}
 P_n(z) =  \sum_{\tau \in \mathcal{T}} \tau(z)^n \tau'(z)(1 + (K_n  \circ \tau) (z)), \quad  z\in D(0,\rho_x),
\end{equation}
where $K_n(z)$ is an analytic function in $D(0,\rho_x)$ such that
\begin{align}\label{estimate-K_n}
K_n(z) = O\left(n^{-1/2}\rho_x^{-2n}\right) 
\end{align}
locally uniformly on $D(0,\rho_x)$ as $n \rightarrow \infty$.  

Equation \eqref{cor-eqn-one} is equivalent to
\begin{equation}\label{cor-eqn-two}
P_n(z) = z^n(1+ K_n(z)) + 
 \sum_{j=1}^s P_n(T_j(z))  T_j '(z), \quad z \in D(0,\rho_x).
 \end{equation}
\end{theorem}

Evaluating \eqref{cor-eqn-two} at $z=0$ yields the curious identity (see \eqref{formula-Tj-1})
\[
P_n(0)=\sum_{j=1}^sr_j^2P_n(c_j).
\]

Theorem \ref{thm-cmcds} tells us that $\lim_{n\to\infty}P_n(z)/z^n=1$ for all $z\in\Delta(0,\rho_a)$, and what happens for $z\in \cj{D(0,\rho_a)}$ has to be deciphered from \eqref{cor-eqn-one}. We will describe the behavior of $P_n$ in $D(0,\rho_a)$ in terms of functions that we now introduce.  

For every $\sigma\in (0,1)$, let 
\begin{align}\label{definition-theta}
\Theta_{\sigma} (t) :=  t \sum_{v \in \Z} \sigma^{v} e^{\sigma^{v} t}, \quad \Re{t}<0,
\end{align}
and for each $j\in \N_s$ with $|a_j|=\rho_a$, and every integer $n\geq 1$, we define the function 
\begin{align*}
F_{j,n}(z):= \sum_{\tau \in \mathcal{T} \setminus \mathcal{T}_j} \frac{\Phi_j '(\tau(z))}{\Phi_j(\tau(z))} \Theta_{\sigma^2_j}(n\alpha_j\Phi_j(\tau(z))) \tau'(z),
\quad  z \in \cj{D(0,\rho_a)}\setminus\{a_j\},
\end{align*}
where 
\begin{align}\label{defTj}
\alpha_j := |a_j|^{-1}- |a_j|,\quad \mathcal{T}_j := \{ T_j \tau : \tau \in \mathcal{T}\}.
\end{align}

We observe that each $\Theta_\sigma$ is multiplicatively periodic, i.e., $\Theta_\sigma(\sigma t)=\Theta_\sigma(t)$. Hence, it is bounded on any cone of the form 
\[
\{t:\pi/2+\epsilon<\arg t< 3\pi/2-\epsilon,\ 0<\epsilon<\pi/2 \}.
\]
Since $\Re(\Phi_j(z))<0$ for every $z\in \cj{D(0,|a_j|)}\setminus\{a_j\}$, it follows that for each $j$, the family of functions $(F_{j,n})_{n\geq 1}$ is bounded on compact subsets of $\cj{D(0,\rho_a)}\setminus \{a_j:|a_j|=\rho_a\}$.

\begin{theorem}\label{thm-cmcds-2}For values of  $z\in \cj{D(0,\rho_a)}$, it happens that
 \begin{equation}\label{behavior-inside-Drhoa}
P_n(z)= z^n+\frac{1}{n} \sum_{j:|a_j|=\rho_a}a_j^{n+1} F_{j,n} (z) + O \left( \frac{\rho_a^n}{n^2}\right)
\end{equation}
uniformly on closed subsets of $\cj{D(0,\rho_a)}\setminus \{a_j:|a_j|=\rho_a\}$ as $n \to \infty$, 
while for every $j$  with $|a_j|=\rho_a$,
\begin{align}\label{behavior-at-points-aj}
P_n(a_j) = \frac{a_j ^n}{1- \sigma_j^2} + O\left( \frac{\rho_a^n}{n}\right).
\end{align}
\end{theorem}

Theorems \ref{thm-asymp-leading}, \ref{thm-cmcds}, and \ref{thm-cmcds-2} will all be deduced from  series representations for $P_n$ and $\kappa_n$  given below in Theorem \ref{thm-poly}. As such, Theorem \ref{thm-poly} could be regarded as the main result of this paper.  
 
We finish this section with a few comments on the zero distribution of $P_n$. If $z_{1,n},\ldots,z_{n,n}$ denote the $n$ zeroes of $P_n$, we let
\[
\nu_n:=\frac{1}{n}\sum_{k=1}^n\delta_{z_{k,n}}, \quad n\geq 1,
\]
where $\delta_z$ is the unit point mass at $z$. The weak-star convergence of the sequence $(\nu_n)$ to the measure $\nu$ (symbolically, $\nu_n\wsc \nu$)  means that for every  function $f$ compactly supported and continuous in $\mathbb{C}$, $\lim_{n\to\infty}\int f d\nu_n=\int fd\nu$. The measure $\nu$ is said to be a weak-star limit point of $(\nu_n)$ if there is a subsequence $(n_k)$ of the natural numbers such that $\nu_{n_k}\wsc \nu$.  

Observe that because of \eqref{exterior-asymptotics}, for any given $r>\rho_a$, $P_n$ will cease to have zeroes on $\Delta(0,r)$ once $n$ is large enough, so that every weak-star limit point of $(\nu_n)$ must be supported on the closed disk $\cj{D(0,\rho_a)}$.  

For $\caliD$ as in \eqref{notation:CMCD} and satisfying Assumption \ref{assumption}, there is always a subsequence $(n_k)$ of the sequence of natural numbers such that 
\begin{align}\label{weakconvergence}
\nu_{n_k}\wsc \nu_{\rho_a},
\end{align}
where $\nu_{\rho_a}$ is the arclength measure of the circle $\T(0,\rho_a)$ divided by $2\pi\rho_a$. To prove this convergence, one may argue just as in the proof of Proposition 3.2 of \cite{SS}, where  \eqref{weakconvergence} is proven for a CMCD bounded by two non-concentric circles, that is, when 
\begin{align}\label{D-two-circles}
\mathcal{D} = D(0,1) \setminus \cj{D(c_1, r_1)}, \quad c_1\not=0.
\end{align}
What is mainly needed  is the ability of locating the first singularities of the kernel $\sum_{n=0}^\infty \cj{p_n(\zeta)}p_n(\cdot)$, which can be done using the representation  \eqref{kernel-relation-pn}-\eqref{defkernel} given below in Section \ref{section-reproducing}. With the arguments of \cite{SS} and Lemma 4.3 of \cite{MSS}, one can also verify that $\nu_{\rho_a}$ is the only weak-star limit point of $(\nu_n)$ that is supported in $\T(0,\rho_a)$.

With the availability of Theorem \ref{thm-cmcds-2},  more can be said when $\mathcal{D}$ is as in \eqref{D-two-circles}. In this case, we get from  \eqref{behavior-inside-Drhoa} that
\begin{equation}\label{behavior-inside-Drhoa-one-disk}
n a_1^{-n-1}P_n(z)=  \frac{\Phi_1 '(z)}{\Phi_1(z)} \Theta_{\sigma_1^2}(n\alpha_1\Phi_1(z)) + O \left( n^{-1}\right)
\end{equation}
locally uniformly on $D(0,\rho_a)$ as $n\to\infty$.

Let $(n_k)$ be a subsequence of the natural numbers and let $\sigma\in (0,1)$. As proven in \cite[Theorem 2]{DMN},  the sequence of functions $(\Theta_\sigma(n_k t))_{k=1}^\infty$ converges normally on $\Re t<0$ if and only if there exists $q\in [0,1)$ such that 
$\log_{\sigma}n_k\to q$ modulo $1$, that is, 
\begin{align}\label{mod1-limit}
 \lim_{k\to \infty} e^{2\pi i (\log_\sigma n_k- q)}=1,
\end{align}
in which case 
\[
\lim_{k\to\infty}\Theta_\sigma(n_kt)=\Theta_\sigma(\sigma^qt)
\]
locally uniformly on $\Re t<0$. Also, for every $q\in [0,1)$, it is possible to find  a subsequence $(n_k)$ for which \eqref{mod1-limit} is true.  By \eqref{behavior-inside-Drhoa-one-disk}, we conclude that the family
\[
\caliF=\left\{ \frac{\Phi_1 '(z)}{\Phi_1(z)} \Theta_{\sigma_1^2}(\sigma_1^{2q}\alpha_1\Phi_1(z)):\ q\in[0,1)\right\} 
\]
comprises all the normal limits that the sequence $(n a_1^{-n-1}P_n)$ has in $D(0,\rho_a)$.

By Helly's theorem \cite[Theorem 0.1.3]{ST}, every subsequence of $(\nu_n)$ has in turn a subsequence, say $(\nu_{n_k})$,  weakly-star converging to some measure $\nu$ supported in $\cj{D(0,\rho_a)}$. By possibly having to move along a subsequence of $(n_k)$, we can assume that the sequence $(n_k a_1^{-n_k-1}P_{n_k})$ converges normally in $D(0,\rho_a)$ to an element of the family $\caliF$. On a fixed disk $D(0,r)$ of radius $r<\rho_a$, this element will have a finite number of zeroes, which by Hurwitz's theorem, is the exact same number of zeroes that every $P_{n_k}$ will have in $D(0,r)$ for all $k$ large enough. Therefore, $\nu_{n_k}(D(0,r))\to 0=\nu(D(0,r))$, which means that $\nu$ is supported on $\T(0,\rho_a)$, and so $\nu=\nu_{\rho_a}$. The conclusion is that for $\mathcal{D}$ as in \eqref{D-two-circles}, the whole sequence $(\nu_n)$ converges to $\nu_{\rho_a}$.

It might be possible to extract a similar conclusion from (1.31) for a CMCD bounded by more than two circles.  In this case, however, the higher complexity of the functions $F_{j,n}$ makes the analysis more involved, and we prefer to address that case in a separate work.

\section{The convergence of $\sum_{\tau\in \caliT}|\tau'|$}
The main purpose of this section is to prove Proposition \ref{thm-converge-cases}. We first establish a few facts that will be of use in later sections as well.

The inverse of $T_j$ will be denoted by $T_j^{-1}$. If $v$ is a positive integer, the composition of $T_j$ with itself a number $v$ of times will be denoted by $T_j^v$, while if $v$ is a negative integer, then $T^{v}_j$ will denote the inverse of $T^{-v}_j$ (i.e., the composition of $T^{-1}_j$ with itself a number $-v$ of times). It follows from the definition \eqref{def-Tj} that 
\[
T_j ^v(z) = \Phi_j^{-1}(\sigma^{2v} \Phi_j(z)),
\quad
v \in \Z,
\]
with the understanding that $T_j^0(z)=T_0(z)\equiv z$.

Note that since $\Phi_j(D(c_j,r_j))=D(0,\sigma_j)$, we have
\begin{align}\label{inclusion}
T_j(D(0,1))\subset  D(c_j,r_j),\qquad j\in \N_s.
\end{align}

\begin{lemma} \label{inclusions}For every $j\in \N_s$, we have the inclusions
 \begin{equation}\label{between}
 T_j(\overline{D(0,r)}) \subset D(0,r), \quad |a_j| < r < 1/|a_j|,
\end{equation}
\begin{equation}\label{between-inverse} 
T_j^{-1}(\cj{\Delta(0,r)})\subset \Delta(0,r), \quad |a_j| < r < 1/|a_j|,
\end{equation}
and
\begin{equation}\label{compact1}
T_j(\overline{D(0,|a_j|)}\setminus\{a_j\})\subset D(0,|a_j|).
\end{equation}
\end{lemma}
\begin{proof}
 We begin by noticing that  the zero $a_j$ and the pole $\cj{a}_j^{-1}$ of $\Phi_j$ are the two fixed points of $T_j$:
\begin{align*}
T_j(a_j)=a_j,\quad T_j(\cj{a}_j^{-1})=\cj{a}_j^{-1},\qquad j\in \N_s.
\end{align*}

Let $r$ be such that $|a_j|<r<1/|a_j|$. Since the zero $a_j$ of $\Phi_j$ belongs to $D(0,r)$ and  the pole $\cj{a}_j^{-1}$ of $\Phi_j$  belongs to $\Delta(0,r)$, we see that $0$  belongs to the interior of the closed disk $\Phi_j(\overline{D(0,r)})$, and so if $s_j(z): = \sigma_j ^2z$, then 
\[
0\in \inte\left[s_j (\Phi_j(\overline{D(0,r)})\right],\quad s_j (\Phi_j(\overline{D(0,r)})) \subset \Phi_j(D(0,r)),
\]  
and
\[
\infty\in\inte\left[ s^{-1}_j (\Phi_j(\overline{\Delta(0,r)})) \right], \quad s^{-1}_j (\Phi_j(\overline{\Delta(0,r)})) \subset \Phi_j(\Delta(0,r)),
\]
where we are using the notation $\inte[A]$ to mean the interior of the set $A$.

Applying $\Phi_j^{-1}$ to both sides of the previous relations we conclude that
\begin{equation}\label{between-1}
a_j\in \inte\left[T_j(\overline{D(0,r)})\right],
\end{equation}
\begin{equation}\label{between-inverse-1}
\cj{a}_j^{-1}\in \inte\left[T_j^{-1}(\cj{\Delta(0,r)})\right],
\end{equation}
and that both \eqref{between} and \eqref{between-inverse} hold true.

Similarly, when  $a_j\not=0$, $\Phi_j(\overline{D(0,|a_j|)})=\cj{B}_j$ and $\Phi_j(a_j)=0$,  where  
\begin{align}\label{disk-Bj}
B_j:=  
\left\{
t: 
\left|
t+ \frac{|a_j|}{1+|a_j|^2}
\right|
< \frac{|a_j|}{1+|a_j|^2}
\right\}.
\end{align}
Clearly, $s_j(\cj{B}_j\setminus\{0\})\subset B_j$,  and  applying $\Phi_j^{-1}$ to  this inclusion yields \eqref{compact1}. That \eqref{compact1} is also true when $a_j=0$ is trivial since in such a case $T_j(z)=\sigma_j^2z$. 
\end{proof}

The family $\mathcal{T}$ has been defined in \eqref{def-T-family} as consisting of the identity $T_0(z)\equiv z$ plus all finite compositions (strings) of transformations $T_j$, $j\in\N_s$. Any two different such strings represent different functions. This follows, for instance, as a corollary of the following lemma.
\begin{lemma} \label{equality-strings}Suppose that $\tau_1=T_{j_n}T_{j_{n-1}}\cdots T_{j_1}$ and $\tau_2=T_{k_m}T_{k_{m-1}}\cdots T_{k_1}$ are 
such that $\tau_1(\caliD)\cap\tau_2(\caliD)\not=\emptyset$. Then $n=m$ and $j_l=k_l$ for all $l\in\N_m$. 
\end{lemma}
\begin{proof}Lets us write
\[
\tau_1= T_{j_n}T_{j_{n-1}}\cdots T_{j_1}T_{j_0},\quad \tau_2=T_{k_m}T_{k_{m-1}}\cdots T_{k_1}T_{k_0}, 
 \]
 where for convenience we have appended to the right of the original strings the identity map, that is, $T_{j_0}=T_{k_0}=T_0$. Let us assume that $m\leq n$ and that $\tau_1(\caliD)\cap\tau_2(\caliD)\not=\emptyset$. By \eqref{inclusion}, $\tau_1(\caliD) \subset D(c_{j_n},r_{j_n}) $ and $\tau_2(\caliD) \subset D(c_{k_m},r_{k_m}) $. Because $D(c_{j_n},r_{j_n})\cap D(c_{k_m},r_{k_m})=\emptyset$ if $j_n\not=k_m$, it follows that $j_n=k_m$, and so $ T^{-1}_{j_n}\tau_1(\caliD)\cap T^{-1}_{k_m}\tau_2(\caliD)\not=\emptyset$, or equivalently,
 \[
T_{j_{n-1}}\cdots T_{j_1}T_{j_0}(\caliD)\cap T_{k_{m-1}}\cdots T_{k_1}T_{k_0}(\caliD) \not=\emptyset
 .\]
 Applying repeatedly the same argument, we conclude that 
$T_{j_{n-l}}=T_{k_{m-l}}$ for all $0\leq l\leq m-1$ and that $T_0(\caliD)\cap T_{j_{n-m}}\cdots T_{j_0}(\caliD)=\emptyset$. Since $T_0(\caliD)=\caliD$, again by \eqref{inclusion}, this is only possible if $n=m$.
\end{proof}

It will be convenient to have the following definition at hand.
\begin{definition}\label{defelln}
For every element $\tau\in \mathcal{T}^*$ we define  $\ell(\tau)$ to be the  number $n$ of operators $T_j$ figuring in the representation $\tau=T_{j_n}T_{j_{n-1}}\cdots T_{j_1}$.
\end{definition}

\subsection{Proof of Proposition \ref{thm-converge-cases}}\label{section3.1}
Let us begin by establishing that  the normal convergence in $D(0,\rho_a^{-1})$ of the series $\sum_{\tau \in \mathcal{T}} |\tau'(z)|$ is equivalent to its convergence at some point of said disk. It is not difficult to see that $T_j(\cj{D(0,\rho_a^{-1})}) \subset \cj{D(0,\rho_a^{-1})}$ for all $j\in\N_s$ (this follows, for instance,  as a consequence of the inclusion \eqref{between} in Lemma \ref{inclusions} above). Since each $\tau\in\caliT^*$ is a composition of $T_j$'s, it follows that $\tau(\cj{D(0,\rho_a^{-1})}) \subset \cj{D(0,\rho_a^{-1})}$ for all $\tau\in \caliT$, so that the pole of every such $\tau$ lies in $\Delta(0,\rho_a^{-1})$.

 Since each $\tau\in\caliT$ is a M\"{o}bius transformation whose pole $p_\tau$ satisfies that $\rho_a^{-1}<p_\tau\leq \infty$, we can write its derivative $\tau'$ in the form 
\[
\tau'(z) = \frac{\gamma_\tau}{(1-z/p_\tau)^2}.
\]
Then, for every  $\rho\in (0,\rho_a^{-1})$ and any collection of points $\{z_\tau\}_{\tau\in \caliT}\subset D(0,\rho)$, we have 
\begin{align}\label{series-sandwich}
\frac{1}{(1+\rho\rho_a)^2}\sum_{\tau\in\caliT}|\gamma_\tau|\leq  \sum_{\tau\in\caliT}|\tau'(z_\tau)|\leq \frac{1}{(1-\rho\rho_a)^2} \sum_{\tau\in\caliT}|\gamma_\tau|.
\end{align}
Hence, the normal convergence of $\sum_{\tau\in\caliT}|\tau'(z)|$ in $D(0,\rho_a^{-1})$ is equivalent to the convergence of $\sum_{\tau\in\caliT}|\gamma_\tau|$, which is equivalent to the convergence of  $\sum_{\tau\in\caliT}|\tau'(z)|$ at some point $z\in D(0,\rho_a^{-1})$.

\emph{Proof of part i).} If all the disks $D(c_j,r_j)$, $j\in \N_s$, have their center $c_j\in (-1,1)$, then  $\tau((-1,1))\subset(-1,1)$ for every $\tau\in \caliT$. Let  $I$ be a closed interval contained in $\caliD\cap (-1,1)$.  By Lemma \ref{equality-strings},  $\{\tau(I)\}_{\tau\in\caliT}$ is a disjoint collection of closed subintervals of $(-1,1)$. Hence,
\[
\sum_{\tau\in\caliT}\int_I|\tau'(x)|dx=\sum_{\tau\in\caliT}\int_{\tau(I)}dx\leq 2. 
\]
By the mean value theorem, for every $\tau\in \caliT$, we can find $x_\tau\in I$ such that $\int_I|\tau'(x)|dx=|\tau'(x_\tau)|\int_{I}dx$, which in view of \eqref{series-sandwich} gives
\begin{align*}
\sum_{\tau\in\caliT}|\gamma_\tau|\leq  (1+\rho_a)^2\sum_{\tau\in\caliT}|\tau'(z_\tau)|\leq \frac{2(1+\rho_a)^2}{\int_{I}dx},
\end{align*}
that is, Assumption \ref{assumption} holds true in this case.

\emph{Proof of part ii).}  Let us assume \eqref{ineq-radii-condition}. From the definition of $T_j$ in \eqref{def-Tj} and the relations \eqref{formulas-for-c-and-r}, we find 
\begin{align}\label{formula-Tj-1}
T_j(z)=\frac{(\sigma_j^2-|a_j|^2)z+a_j(1-\sigma_j^2)}{1-|a_j|^2\sigma_j^2-\cj{a}_j(1-\sigma_j^2)z}=\frac{(r_j^2-|c_j|^2)z+c_j}{1-\cj{c}_jz}.
\end{align}
Hence, $T_j '(z)  =  
r_j^2( 1-\cj{c}_jz)^{-2}$, and since $|c_j|\leq |a_j|\leq \rho_a$, we have 
\begin{align} \label{derivativeTjbound}
|T_j'(z)|\leq r_j^2(1-|c_j|\rho_a)^{-2},\quad |z|\leq \rho_a,\quad j\in \N_s.
\end{align}

If $\tau\in\caliT^*$ is of length $\ell(\tau)=n$, then $\tau=T_{j_n}T_{j_{n-1}}\cdots T_{j_1}$ and 
\[
|\tau'(z)|=|T'_{j_n}(\tau_1(z))||T'_{j_{n-1}}(\tau_2(z))|\cdots|T'_{j_2}(\tau_{n-1}(z))||T'_{j_1}(z)|
\]
where $\tau_k=T_{j_{n-k}}\cdots T_{j_1}$, $k=1,\ldots, n-1 $. Since $\tau(\cj{D(0,\rho_a)})\subset \cj{D(0,\rho_a)}$ for every $\tau\in \caliT$, it follows from \eqref{derivativeTjbound}
that if $\ell(\tau)=n$, then
\[
|\tau'(z)|\leq  \prod_{k=1}^n\frac{r_{j_k}^2}{(1-|c_{j_k}|\rho_a)^{2}}. 
\]
Therefore, 
\[
 \sum_{\tau\in\caliT^*}|\tau'(z)|=\sum_{n=1}^{\infty}\sum_{\tau:\ell(\tau)=n}|\tau'(z)|\leq \sum_{n=1}^{\infty}\left(\sum_{j=1}^s\frac{r_{j}^2}{(1-|c_{j}|\rho_a)^{2}}\right)^n <\infty.
\]

\emph{Proof of part iii).} Let $\mathcal{D}$ (as given by \eqref{notation:CMCD}) and 
\begin{equation*}
\wt{\mathcal{D}} := D(0,1) \setminus  \bigcup_{j=1} ^s \cj{D(\tilde{c}_j, \tilde{r}_j) }
\end{equation*}
be two CMCDs such that $\mathcal{D}=\Psi(\wt{\mathcal{D}})$ for some automorphism $\Psi$ of the unit disk $D(0,1)$. We assume the labeling is such that 
\[
 \Psi(D(\tilde{c}_j, \tilde{r}_j))=D(c_j, r_j),\qquad j\in\N_s.
\]
Let the maps $T_j$ and the family $\caliT^*$ associated to $\mathcal{D}$ be as introduced in \eqref{def-Tj}-\eqref{def-T*-family}, and let $\wt{T}_j$ and $\wt{\caliT}^*$ denote those corresponding to $\wt{\mathcal{D}}$. Since $\Phi_j$ takes $D(c_j,r_j)$ onto $D(0,\sigma_j)$, we see that $\wt{\Phi}_j=\Phi_j\circ\Psi$ takes $D(\tilde{c}_j, \tilde{r}_j)$ onto $D(0,\sigma_j)$. Hence,
\[
 \wt{T}_j(w)=\wt{\Phi}^{-1}_j(\sigma_j^2\wt{\Phi}_j(w))=(\Psi^{-1}\circ T_j\circ\Psi)(w)
\]
which implies that 
\[
 \wt{\caliT}^*=\{\Psi^{-1}\circ \tau \circ\Psi:\tau\in \caliT^*\}, \quad \caliT^*=\{\Psi\circ \tilde{\tau} \circ\Psi^{-1}:\tilde{\tau}\in \wt{\caliT}^*\}.
\]

Since $\Psi'$ and $(\Psi^{-1})'$ are bounded in $D(0,1)$, it follows that $ \sum_{\tau\in \caliT^*}|\tau'(z)|$ converges for some $z\in D(0,1)$ if and only if $\sum_{\tilde{\tau}\in \wt{\caliT}^*}|\tilde{\tau}'(w)|$ converges for some $w\in D(0,1)$.

\section{The reproducing kernel $\mathcal{K}_{\mathcal{D}}$}\label{section-reproducing}

For a bounded domain $D$, let us denote by $L^2_D$ the linear space of functions that are analytic and square integrable on $D$. Let the inner product $\langle \cdot, \cdot\rangle_D$ and its associated norm be given by 
\begin{equation*}
\langle f, g \rangle_D
=
\int_{D} f(z)\overline{g(z)}  dA(z),\quad \|f\|_D=\sqrt{\langle f, f \rangle_D},\quad f,\,g\in L^2_D.
\end{equation*}
 
The space $L^2_D$ under $\langle \cdot, \cdot\rangle_D$ is a Hilbert space, and the polynomials form a complete subspace of $L^2_D$ if $D$ is, for instance, the interior of a Jordan curve  \cite[Chap. I]{Gaier}. 

The Bergman kernel $K_D(z,\zeta)$ for $D$ is a function analytic in $z$ and anti-analytic in $\zeta$, that is uniquely determined by the reproducing property
\[
 f(z)=\int_D f(\zeta)K_D(z,\zeta)dA(\zeta), \qquad z\in D,\ f\in L^2_{D},
\]
and if $D'$ is a bounded domain and $\Phi:D'\to D$ is a conformal map of $D'$ onto $D$, then 
\begin{align}\label{relating-kernels}
 K_{D'}(z,\zeta)=\Phi'(z)\cj{\Phi'(\zeta)}K_D(\Phi(z),\Phi(\zeta)),
\end{align}
see, e.g.,  \cite[Chap. I,  \S 5]{Gaier}.

 Let $\caliD$ be a  CMCD. Because of the equivalence of the norms $\|\cdot\|_\caliD$ and  $\|\cdot\|_{D(0,1)}$, the space $L^2_{D(0,1)}$, when endowed with the inner product $\langle \cdot,  \cdot\rangle_\mathcal{D}$, forms
a Hilbert space, and point evaluation functionals acting on $L^2_{D(0,1)}$ (under the $\|\cdot\|_\caliD$ norm) happen to be bounded.

The Riesz representation theorem then guarantees the existence of a unique function $\mathcal{K}_{\caliD}(z,\zeta)$  defined for $z$ and $\zeta$ in $D(0,1)$, analytic in $z$ and anti-analytic in $\zeta$, characterized by the reproducing property
\begin{equation*}
f(z) = \int_{\caliD} f(\zeta)  \mathcal{K}_{\caliD}(z,\zeta)  dA(\zeta),
\quad
z\in D(0,1),\ f\in L^2_{D(0,1)}.
\end{equation*}

Because the polynomials form a complete subspace of $L^2_{D(0,1)}$ under $\|\cdot\|_{D(0,1)}$, they also form a complete subspace of $L^2_{D(0,1)}$ under $\|\cdot\|_\caliD$, and so we have 
\begin{align}\label{kernel-relation-pn}
\mathcal{K}_\caliD(z,\zeta)=\sum_{n=0}^\infty p_n(z,\caliD)\cj{p_n(\zeta,\caliD)}, \quad z,\zeta \in D(0,1).
\end{align}

If $\caliD$ has the simpler form \eqref{annulus}, the corresponding orthonormal polynomials $p_n$ are given by  \eqref{PnSingleDisk}, so that in this case\begin{align*}
\mathcal{K}_\caliD(z,\zeta)={} &\sum_{n=0}^\infty \frac{(n+1)(z\cj{\zeta})^n}{1-r_1^{2n+2}}
=\frac{1}{(1-z\cj{\zeta})^2}+r_1^2\mathcal{K}_\caliD(r_1 z,r_1 \zeta)\\
={}& \sum_{j=0}^\infty \frac{r_1^{2j}}{(1-r_1^{2j}z\cj{\zeta})^2}.
\end{align*}

The following proposition extends this formula to an arbitrary CMCD. It will play an important  role in the construction of the series expansion for $P_n$ carried out in the next section.

\begin{proposition} 
If $\mathcal{D}$ is a CMCD for which Assumption \ref{assumption} holds, then we have the representation 
\begin{align}\label{defkernel}
\mathcal{K}_{\caliD}(z,\zeta)  = 
\sum_{\tau \in \mathcal{T}} 
\frac{ \tau'(z) }{(1-\tau (z) \cj{\zeta})^2}, \quad z,\zeta \in D(0,1).
\end{align}
\end{proposition}
\begin{proof}
The Bergman kernel for the unit disk is $(1- z \overline{\zeta})^{-2}$, so that 
for every $f\in  L^2_{D(0,1)}$ and $z\in D(0,1)$, 
\begin{align}\label{rep-u-disk}
f(z) ={} & \int_{D(0,1)}  \frac{f(\zeta)}{(1- z \overline{\zeta})^2}dA(\zeta).
\end{align}
Since each $\Phi_j$ is an automorphism of the unit disk, we get from  \eqref{relating-kernels} that for every $j\in\N_s$,
 \begin{equation}\label{magick}
\frac{1}{(1- z \overline{\zeta})^2} = \frac{ \Phi_j'(z) \overline{\Phi_j'(\zeta)} }{(1-  \Phi_j(z) \overline{\Phi_j(\zeta)} )^2}, \quad z, \zeta \in D(0,1),
\end{equation}
and for the same reason, since $\sigma_j^{-1}\Phi_j$ is a conformal map of $D(c_j,r_j)$ onto $D(0,1)$, it is the case that
\begin{align}\label{kernel-Dcr}
K_{D(c_j,r_j)}(z,\zeta)=\frac{ \sigma_j^2 \Phi_j ' (z) \overline{\Phi_j'(\zeta)} }{(\sigma_j^2-\Phi_j(z) \overline{\Phi_j(\zeta)})^2}.
\end{align}
 
We now show that for each $j\in\N_s$ and $f\in L^2_{D(0,1)}$, 
\begin{align}\label{integral-over-Dcr}
\int_{D(c_j,r_j)}  \frac{ f(\zeta)  }{(1- z\overline{\zeta})^2  }  dA(\zeta) = T'_j ( z) f (T_j (z)),\quad z\in D(0,1).
\end{align}

By relation \eqref{magick}, the fact that $\Phi_j(T_j(z))=\sigma_j^2 \Phi_j(z)$, and \eqref{kernel-Dcr},   we have  
\begin{align*}
\int_{D(c_j,r_j)}  \frac{ f(\zeta)   dA(\zeta)}{(1- z \overline{\zeta})^2  } ={} & \int_{D(c_j,r_j)} f(\zeta)\frac{\Phi_j' ( z) \overline{\Phi_j'(\zeta)}}{(1- \Phi_j(z) \overline{\Phi_j(\zeta)})^2} dA(\zeta)
\\
={} & T_j'(z)\int_{D(c_j,r_j)}  \frac{f(\zeta)\sigma_j^2  \Phi_j' (T_j ( z)) \overline{\Phi_j'(\zeta)}}{(\sigma_j^2- \Phi_j( T_j (z)) \overline{\Phi_j(\zeta)})^2}
 dA(\zeta)\\
 ={} &  T_j'(z)\int_{D(c_j,r_j)} f(\zeta) K_{D(c_j,r_j)}(T_j (z),\zeta)
 dA(\zeta)\\
 ={} & T'_j (z) f (T_j (z)),
\end{align*}
the latter equality being valid since $T_j(z) \in D(c_j,r_j)$ any time $z\in D(0,1)$.

Having established the above facts, it is now easy to prove \eqref{defkernel}. First, note that 
because of the unique representation that each $\tau\in\caliT$ has as a composition of $T_j$'s, we can write
\begin{align}\label{decomposition-family-T}
 \mathcal{T}^*=\bigcup_{j=1}^s\{T_j\tau:\tau\in \mathcal{T}\}=\bigcup_{j=1}^s\{\tau T_j:\tau\in \mathcal{T}\}
 \end{align}
(these unions being disjoint). Since $\tau(D(0,1))\subset D(0,1)$ for every $\tau\in D(0,1)$, we can use \eqref{rep-u-disk}, \eqref{integral-over-Dcr}, and the  decomposition \eqref{decomposition-family-T} to compute, for every $z\in \caliD$,
\begin{align*}
&\int_{\mathcal{D}}f(\zeta)  \left(\sum_{\tau \in \mathcal{T}}\frac{\tau'(z)}{(1- \tau(z) \overline{\zeta})^2  }\right) dA(\zeta) \\
 &=
 \sum_{\tau \in \mathcal{T}}\int_{D(0,1)}  \frac{ f(\zeta)\tau'(z)}{(1- \tau(z) \overline{\zeta})^2  }  dA(\zeta)  - \sum_{j \in \Lambda_s}\sum_{\tau \in \mathcal{T}}\int_{D(c_j,r_j)} \frac{ f(\zeta) \tau'(z)}{(1- \tau(z) \overline{\zeta})^2  }  dA(\zeta)
\\
&=
\sum_{\tau \in \mathcal{T}} \tau'(z) f(\tau(z))  - \sum_{j \in \Lambda_s}  \sum_{\tau \in \mathcal{T}}(T_j\circ \tau)'(z) f (T_j (\tau(z)))
\\
&=   \sum_{\tau \in \mathcal{T}} \tau'(z)  f( \tau(z))  - \sum_{\tau \in \mathcal{T}^*} \tau'(z)  f( \tau(z))  = f(z),
\end{align*}
the latter equality being true because the only function in $\mathcal{T} \setminus \mathcal{T}^*$ is the identity function. Thus, \eqref{defkernel} is established.  
\end{proof}

\section{Series representation for $P_n$}

Recall that we have defined the function $m:[0,\rho_a^{-1}]\to \R$ by \eqref{defmandMr}. This function is 
well-defined and continuous on $[0,\rho_a^{-1}]$, since the pole of each $T_j$ is contained in $\Delta(0,\rho_a^{-1})$. The composition of $m(r)$ with itself a number $v$ of times will be denoted by $m^v(r)$.

Some properties of the function $m(r)$ are summarized in Lemma  \ref{lemma1} below. Recall Definition \ref{defelln}, where the length $\ell(\tau)$ of a transformation $\tau\in\caliT^*$ was introduced.

\begin{lemma} \label{lemma1} (i) $m(r)$ is a strictly increasing function that maps $[\rho_a,\rho_a^{-1}]$ onto $[\rho_a,\rho_a^{-1}]$ and satisfies  
\begin{align}\label{ineq:mandM1}
\rho_a< m(r)<r, \qquad r\in(\rho_a,\rho_a^{-1}).
\end{align}

(ii) For every $r\in[0,\rho_a^{-1}]$ and $\tau\in \mathcal{T}^*$,
\begin{align}\label{ineq:mandaM3}
 \tau (\cj{D(0,r)})\subset \cj{D(0,m^\ell(r))},\qquad (where\ \ell=\ell(\tau)).
\end{align}

(iii)
\begin{align}\label{maxproperty}
m(\rho_x)/\rho_x=\rho_x^{-2}=\min_{r\in[\rho_a,\rho_a^{-1}]}m(r)/r.
\end{align}

\end{lemma}

\begin{proof}
By the very definition of $m(r)$, $\cj{D(0,m(r))}$ is the smallest closed disk about the origin that contains each of the closed disks $T_j(\cj{D(0,r)})$, so that 
\begin{equation}\label{compact3}
T_j(\cj{D(0,r)})\subset \cj{D(0,m(r))}, \quad r\in[0,\rho_a^{-1}],\  j\in \N_s.
\end{equation}

If $r\in(\rho_a,\rho_a^{-1})$, then \eqref{between} holds true for every $j\in \N_s$, so that  $\cj{D(0,m(r))}\subset D(0,r)$, which together with  \eqref{between-1} yields \eqref{ineq:mandM1}. 

If $\cj{D(0,r')}\subset D(0,r)$, then for all $j\in \N_s$, the closed disk $T_j(\cj{D(0,r')})$ is contained in $D(0,m(r))$, and so $\cj{D(0,m(r'))}\subset D(0,m(r))$, proving that $m$ is strictly increasing.

We now show that $m$ maps $(\rho_a,\rho_a^{-1})$ onto $(\rho_a,\rho_a^{-1})$. Fix $r\in(\rho_a,\rho_a^{-1})$. From \eqref{between-inverse} and \eqref{between-inverse-1}, it follows that 
\begin{align*}
r<M(r):=\min_{j\in \N_s,\,z\in \cj{\Delta(0,r)}}
|T^{-1}_j(z)|<\rho_a^{-1},
\end{align*}
so that 
\begin{align*}
\cj{\Delta(0,M(r))}\supset T^{-1}_j(\cj{\Delta(0,r)}),\qquad j\in \N_s,
\end{align*}
and there is at least one of the maps $T^{-1}_j$, say $T_1^{-1}$, such that $ T^{-1}_1(\cj{\Delta(0,r)})$ touches the circle $\T(0,M(r))$. Hence $T_j(\cj{D(0,M(r))})\subset\cj{D(0,r)}$ for all $j\in \N_s$ and $T_1(\cj{D(0,M(r))})\cap \T(0,r)\not=\emptyset$, which in view of the definition of $m(r)$ means that  $m(M(r))=r$, and so the map $m$ is onto.

By continuity, $m:[\rho_a,\rho_a^{-1}]\to[\rho_a,\rho_a^{-1}]$ is surjective as well, with endpoint values  
\begin{align}\label{endpointvalues}
m(\rho_a)=\rho_a,\qquad m(\rho_a^{-1})=\rho^{-1}_a.
\end{align} 

The maximum value of $m$ in $[0,\rho_a^{-1}]$ is $\rho_a^{-1}$, so that the composition $m(m(r))$ is well-defined. Since  every $\tau\in\caliT^*$ is a composition of a number $\ell=\ell(\tau)$ of $T_j$ maps,  \eqref{ineq:mandaM3} follows by iterations of \eqref{compact3}.  

We now prove \eqref{maxproperty}. From \eqref{formula-Tj-1} we get
\begin{align}\label{defmj}
 m_j(r):=\max_{|z|=r}|T_j(z)|=\frac{r_j^2r}{1-|c_j|r}+|c_j|,\quad r\in [0,\rho_a^{-1}],\  j\in\N_s.
\end{align}
By taking derivatives, it follows  that $m_j$ is strictly increasing and that 
\[
 \min_{r\in [0,\rho_a^{-1}]}m_j(r)/r=m_j(|x_j|)/|x_j|=|x_j|^{-2},
\]
where $x_j$ is given by \eqref{def-of-xj-yj}. Let $j_0\in\N_s$ be such that $|x_{j_0}|=\rho_x$ (recall that by definition $\rho_x=\min_{j\in\N_s}|x_j|$). Then, for every $r\in[0,\rho_a^{-1}]$,
\begin{align}\label{maxrel}
\frac{m(r)}{r}=\max_{j\in\N_s}\frac{m_j(r)}{r}\geq \frac{m_{j_0}(r)}{r}\geq \frac{m_{j_0}(|x_{j_0}|)}{|x_{j_0}|}=\rho_x^{-2}.
\end{align}
But also, by the monotonicity of $m_j$, we see that for every $j\in\N_s$,
\[
 m_{j_0}(|x_{j_0}|)=\rho_x^{-1}\geq |x_j|^{-1}=m_j(|x_j|)\geq m_j(|x_{j_0}|).
 \]
Dividing this relation by $|x_{j_0}|$ we get $\rho_x^{-2}=m(|x_{j_0}|)/|x_{j_0}|$. This and \eqref{maxrel} yield \eqref{maxproperty}.

For later use, we observe that since $m_j$ is strictly increasing, the same argument can be used to prove that  $\rho_x^{-1}>m_j(\rho_x)$ for all $j$  such that $|x_j|>\rho_x$.
 Therefore, 
\begin{align}\label{vrhox}
v(\rho_x):= {} &\max\{m_j(\rho_x)/\rho_x: \ j\in \N_s, \ |x_j|>\rho_x\}<\rho_x^{-2}.
\end{align}

\end{proof}

For every  $r\in (0,\rho_a^{-1})$, define
 \[
\mu(r):=\sup_{|z|\leq r}\sum_{\tau \in \mathcal{T}}|\tau'(z)|.
\]

\begin{lemma} \label{estimate-sum-over-tau_v} 
(i) For every $r\in [\rho_a,\rho_a^{-1})$  and positive integers $n, v$, we have 
\begin{align}\label{estimate-derivative-f_n,1}
\sum_{\tau:\, \ell(\tau)\geq v}\left|\tau'(z)\tau(z)^{n-1}\right| 
\leq   \mu(r)(m^{v}(r))^{n-1},\quad z\in\cj{D(0,r)}. 
\end{align}
(ii) For every $r\in (\rho_a,\rho_a^{-1})$ and every integer $v\geq 1$, let $N=N(r,v)$ be such that 
\[
n\left(m^{v+1}(r)/m^{v}(r)\right)^{n-1}<1
\]
whenever  $n>N$. Then, for all $n>N$ and $z\in\cj{D(0,r)}$,
\begin{align}\label{inequality-for-modulus-f_n1}
\sum_{\tau:\, \ell(\tau)\geq v} \left|\tau(z)^{n}-\tau(0)^{n}\right|\leq r( 2s^v+\mu(r))(m^v(r))^{n-1}.
\end{align}
\end{lemma}
\begin{proof}By  \eqref{ineq:mandM1}-\eqref{endpointvalues}, we have $\rho_a\leq m^\ell(r)\leq m^v(r)$ for all $r\in [\rho_a,\rho_a^{-1})$ and $\ell\geq v$. Then,  by \eqref{ineq:mandaM3}, $ \max_{|z|\leq r}|\tau(z)|\leq m^{v}(r)$ whenever $\ell(\tau)\geq v$, so that 
 \begin{align*}
\sum_{\tau:\, \ell(\tau)\geq v}|\tau'(z)||\tau(z)|^{n-1} 
\leq   \mu(r)(m^{v}(r))^{n-1},\quad z\in\cj{D(0,r)},
\end{align*}
proving \eqref{estimate-derivative-f_n,1}.

Let us now assume that $r$, $v$, and $n$ are as specified in part (ii). Let us write

\[
h_{n,v}(z):=\sum_{\tau:\, \ell(\tau)\geq v} \left|\tau(z)^{n}-\tau(0)^{n}\right|.
\]
Then, for $z\in \cj{D(0,r)}$,
 \begin{align*}
h_{n,v}(z)\leq n\int_0^z\sum_{\tau:\, \ell(\tau)\geq v}\left|\tau'(\zeta)\tau(\zeta)^{n-1}\right|| d\zeta|
\leq nr\mu(r)(m^{v}(r))^{n-1}.
\end{align*}
 To refine this estimate, we note that there are $s^v$ elements $\tau$ with  $\ell(\tau)= v$,  so that 
\begin{align*}
\sum_{\tau:\, \ell(\tau)= v} \left|\tau(z)^{n}-\tau(0)^{n}\right|
\leq {} &2s^v(m^v(r))^{n},
\end{align*}
and since $h_{n,v}(z)=\sum_{\tau:\, \ell(\tau)= v} |\tau(z)^{n}-\tau(0)^{n}| +h_{n,v+1}(z)$, we conclude that 
\begin{align}\label{ineq-first-case}
\left|h_{n,v}(z)\right|\leq {}&2s^v(m^v(r))^{n}+nr\mu(r)(m^{v+1}(r))^{n-1}.
\end{align} 
Since we have that $m^v(r)<r$, the inequality \eqref{inequality-for-modulus-f_n1}
follows at once from \eqref{ineq-first-case}.
\end{proof}

We now recursively define, for every integer $n\geq 0$, a sequence of functions $(f_{n,k})_{k=0}^\infty$ analytic in $D(0,\rho_a^{-1})$ as follows. 

For $k=0$, we set 
\[
f_{n,0}\equiv 1.
\] 

If for some $k\geq 0$, $f_{n,2k}$ has been defined as an analytic function in $D(0,\rho_a^{-1})$, then we set 
\begin{align}\label{f2k+1cont2}
f_{n,2k+1}(z):= &\sum_{\tau \in \caliT^*}\left(f_{n,2k} (\tau(z)) \tau(z)^{n+1}-f_{n,2k} (\tau(0)) \tau(0)^{n+1}\right),\quad |z|<\rho_a^{-1}.
\end{align}

Having defined $f_{n,2k+1}$, we then define $f_{n,2k+2}(z)$ for  $z\in D(0,\rho_a^{-1})$ by choosing some $r$ with $|z|<r<\rho_a^{-1}$ and letting 
\begin{align}\label{functionsf1}
f_{n,2k+2}(z) :=- \frac{1}{2 \pi i} \ointctrclockwise\limits_{\T(0,r)} f_{n,2k+1}(\zeta) \frac{\zeta^{-n-1}}{\zeta - z }  d \zeta,
\quad 
z\in D(0,r).
\end{align}
By Cauchy's theorem, it is clear that $f_{n,2k+2}$ is well-defined and analytic in $D(0,\rho_a^{-1})$

For future use, we  notice the identity (valid for $k\geq 0$)
\begin{align}\label{other-important-identity}
\begin{split}
f_{n,2k+2}(z)+\frac{z}{n+1}f'_{n,2k+2}(z)= -\frac{(n+1)^{-1}}{2 \pi i} \ointctrclockwise\limits_{\T(0,r)} \frac{ f'_{n,2k+1}(\zeta)\zeta^{-n}}{\zeta - z }  d \zeta.
\end{split}
\end{align}  
This follows from \eqref{functionsf1} by integration by parts and the fact that \[ \ointctrclockwise\limits_{\T(0,r)} \left(f_{n,2k+1}(\zeta)\zeta^{-n-1}\right)'d \zeta=0.\] To wit, 
\begin{align*}
\frac{z}{n+1}f'_{n,2k+2}(z)={} 
& -\frac{z}{(n+1)2 \pi i} \ointctrclockwise\limits_{\T(0,r)} \frac{ f_{n,2k+1}(\zeta)\zeta^{-n-1}}{(\zeta - z)^2 }  d \zeta\\
={} &-\frac{1}{(n+1)2 \pi i} \ointctrclockwise\limits_{\T(0,r)} \frac{\zeta \left(f_{n,2k+1}(\zeta)\zeta^{-n-1}\right)'}{\zeta - z}  d \zeta\\
={} &-f_{n,2k+2}(z)-\frac{1}{(n+1)2 \pi i} \ointctrclockwise\limits_{\T(0,r)} \frac{ f'_{n,2k+1}(\zeta)\zeta^{-n}}{\zeta - z }  d \zeta.
\end{align*} 

For every integer $n\geq 0$ and $r\in (\rho_a,\rho_a^{-1})$, let
\begin{align}\label{def-V}
V(r,n):=\frac{r(2s+\mu(r))\left(m(r)/r\right)^n}{r-m(r)}.
\end{align}

\begin{lemma}\label{lemma-estimates}For every $r\in(\rho_a,\rho_a^{-1})$ there exists an index $N_r$ such that  the inequalities
\begin{align}
| f'_{n,2k-1}(z)|\leq{} &(n+1)\mu(r)m(r)^{n}V(r,n)^{k-1},\quad |z|\leq r,\label{ineq:first-2}\\
|f_{n,2k-1}(z)| \leq{} &(r-m(r))r^n V(r,n)^{k},\quad |z|\leq r,\label{ineq:second}\\
| f_{n,2k}(z)|\leq{} &\frac{(r-m(r))V(r,n)^{k}}{||z|-r|},\quad |z|< r,\label{ineq:first-1}
\end{align}
and 
\begin{align}
\sum_{\tau:\, \ell(\tau)\geq v}\left|\left(f_{n,2k} (\tau(z)) \tau(z)^{n+1}\right)'\right|
\leq{} &(n+1) \mu(r)(m^v(r))^nV(r,n)^{k},\quad |z|\leq r,\label{induction-case-1}
\end{align}
hold true for all $n>N_r$, $k\geq 1$, and, in the case of \eqref{induction-case-1}, for all $v\geq 1$.

\end{lemma}

\begin{proof}

We will proceed by induction on $k$, that is, checking first that \eqref{ineq:first-2} and \eqref{ineq:second} are true for $k=1$, and that if  \eqref{ineq:first-2}-\eqref{ineq:second}  are true for some $k=k_0\geq 1$, then so are \eqref{ineq:first-1} and \eqref{induction-case-1}. Then, we show that if  \eqref{ineq:first-1} and \eqref{induction-case-1} are true for $k=k_0$, then \eqref{ineq:first-2} and \eqref{ineq:second} are true for $k=k_0+1$, completing the induction cycle. 

Let $N_r$ be so large that 
\begin{align}\label{index-assumption}
 (n+1)(m^2(r)/m(r))^n<1
\end{align}
once $n>N_r$. If $k=1$,  the inequalities \eqref{ineq:first-2} and  \eqref{ineq:second} follow from \eqref{estimate-derivative-f_n,1} and \eqref{ineq-first-case} (case $v=1$). 

Let us suppose now that \eqref{ineq:first-2} and \eqref{ineq:second} hold true for some $k=k_0\geq 1$. From \eqref{functionsf1} we get 
\begin{align*}
|f_{n,2k_0}(z)|\leq{} &  \frac{r^{-n}}{ ||z|-r| }\max_{|z|=r}|f_{n,2k_0-1}(z)|.
\end{align*}
This and the  inequality \eqref{ineq:second}, which is assumed to be valid for $k=k_0$, prove that  \eqref{ineq:first-1} holds true for $k=k_0$.

 From \eqref{other-important-identity} we get that for $|z|\leq r$,
 \begin{align*}
\left|\left(f_{n,2k_0} (\tau(z)) \tau(z)^{n+1}\right)'\right|
={}& \left|\frac{\tau'(z) \tau(z)^{n}}{2 \pi } \ointctrclockwise\limits_{\T(0,r)} \frac{ f'_{n,2k_0-1}(\zeta)\zeta^{-n}}{\zeta - \tau(z) }  d \zeta\right|\\
\leq {} &(n+1)|\tau'(z)|(m^v(r))^n\frac{r\mu(r)\left(\frac{m(r)}{r}\right)^{n}V(n,r)^{k_0-1}}{r-m^v(r)}\\
\leq{} &(n+1)|\tau'(z)| (m^v(r))^nV(n,r)^{k_0},
\end{align*}
whence the validity of \eqref{induction-case-1} for $k=k_0$ follows at once, and moreover, by the definition \eqref{f2k+1cont2}, if we set $v=1$ in \eqref{induction-case-1} we get that \eqref{ineq:first-2} also holds true for $k=k_0+1$. 

Finally, to show that \eqref{ineq:first-1} holds true for $k=k_0+1$, we combine \eqref{f2k+1cont2}, the fact that \eqref{ineq:first-1} is true for $k=k_0$, the inequality \eqref{induction-case-1} with $v=2$, and \eqref{index-assumption} to get  
\begin{align*}
|f_{n,2k_0+1}(z)|\leq {} &\left|\sum_{k=1}^s(f_{n,2k_0}(T_j(z))T_j(z)^{n+1}-f_{n,2k_0}(T_j(0))T_j(0)^{n+1}\right|\\
& +\left|\int_0^z\sum_{\tau:\, \ell(\tau)\geq 2}\left(f_{n,2k_0} (\tau(\zeta)) \tau(\zeta)^{n+1}\right)'d\zeta\right|\\
\leq {} & 2s(m(r))^{n+1}V(r,n)^{k_0}+r(n+1) \mu(r)(m^2(r))^nV(n,r)^{k_0}\\
\leq {} &(m(r))^nV(r,n)^{k_0}\left(2sm(r)+r\mu(r)(n+1)\left(\frac{m^2(r)}{m(r)}\right)^n\right)\\
\leq {} &(m(r))^nV(r,n)^{k_0}r\left(2s+\mu(r))\right)= (r-m(r))r^nV(r,n)^{k_0+1}.
\end{align*}
\end{proof}

In the next theorem and in any subsequent discussion, $P_n$ and  $\kappa_n$ are the orthogonal polynomial and leading coefficient corresponding to a CMCD  $\caliD$.

\begin{theorem}  \label{thm-poly}  For every $\rho\in (0,\rho_a^{-1})$, there exists an integer $N_\rho$ such that for all $n>N_\rho$, we have 
\begin{align*}
 P_n (z)= \frac{1}{n+1}\caliP_n'(z),
\end{align*} 
where
\begin{align}\label{defcaliP}
 \caliP_n (z):= 
 z^{n+1} \sum _{k=0} ^\infty f_{n, 2k} (z)+\sum _{k=0} ^\infty f_{n, 2k+1} (z), \quad z\in D(0,\rho).
\end{align} 
Also,
\[
(n+1) \kappa_n ^{-2} = \sum_{k=0}^\infty f_{n, 2k} (0).
\]
\end{theorem}
\begin{proof}
We can obviously assume that $\rho_a<\rho<\rho_a^{-1}$. According to Lemma \ref{lemma-estimates}, there is an index $N_\rho$ such that for all $n>N_\rho$, we have 
\begin{align*}
|f_{n,2k+1}(z)|\leq
       (\rho-m(\rho))\rho^n V(\rho,n)^{k+1}, \quad  |z|\leq \rho,\quad k\geq 0,
\end{align*}
\begin{align*}
|f_{n,2k}(z)| \leq \frac{(\rho-m(\rho)) V(\rho,n)^{k}}{||z|-\rho  |} ,\quad 
|z|<\rho,\quad k\geq 1.
\end{align*}

Since $m(\rho)<\rho$, we can find an integer $\tilde{N}_\rho>N_\rho$ such that for all $n>\tilde{N}_\rho$ (see \eqref{def-V})
$
V(\rho,n)<1$, 
which implies that, for all $n>\tilde{N}_\rho$, the two series in the right-hand side of \eqref{defcaliP} converge absolutely and normally on $D(0,\rho)$.  

Let us now fix numbers $\rho, \rho', \rho''$, with $\rho_a<\rho<\rho'<\rho''<\rho_a^{-1}$. We just showed that the analytic functions (subscripts $e$ and $o$ for even, odd)
\[
 f_{n,e}(z):=\sum _{k=0} ^\infty f_{n, 2k} (z),\qquad f_{n,o}(z):=\sum _{k=0} ^\infty f_{n, 2k+1} (z) 
\]
are well-defined on $|z|<\rho''$ for all $n$ larger than some number $\tilde{N}_{\rho''}$, and by definition
 \[
\caliP_n (z)= 
z^{n+1}f_{n,e}(z)+f_{n, o} (z), \quad |z|<\rho'', \quad n>N_{\rho''}.
 \]

By \eqref{functionsf1} and Cauchy's integral theorem, we have that for $\rho'<|z|<\rho''$,
\begin{align*}
f_{n,e}(z)-1 ={} & - \frac{1}{2 \pi i} \ointctrclockwise\limits_{\T(0,\rho'')} f_{n,o}(\zeta) \frac{\zeta^{-n-1}}{\zeta - z }  d \zeta\\
={} &-  f_{n,o}(z)-\frac{1}{2 \pi i} \ointctrclockwise\limits_{\T(0,\rho')} f_{n,o}(\zeta) \frac{\zeta^{-n-1}}{\zeta - z }  d \zeta.
\end{align*}
Hence, 
\begin{align}\label{P_n-for-z-large}
\caliP_n (z)= 
 z^{n+1}-\frac{z^{n+1}}{2 \pi i} \ointctrclockwise\limits_{\T(0,\rho')}f_{n,o}(\zeta) \frac{\zeta^{-n-1}}{\zeta - z }  d \zeta, \quad \rho'<|z|<\rho''.
 \end{align}
It follows that $\caliP_n$ has an analytic continuation to the whole complex plane, and that 
\begin{align*}
\lim_{z\to\infty}\frac{\mathcal{P}_n (z)}{z^{n+1}} &=1.
\end{align*}
By Liouville's theorem, $\caliP_n$ is a monic polynomial of degree $n+1$, so that  $(n+1)^{-1}\caliP_n'(z)$ is a monic polynomial of degree $n$.

To finish the proof it  suffices to show that 
\begin{align*}
\int _{\caliD} \mathcal{P}'_{n} (z)  \overline{z}^m  d A(z)&=
\begin{cases}
    0,& 0 \leq m< n,\\
 f_{n, e} (0), & m=n,
\end{cases}
\end{align*}
since $\kappa_n^{-2}=\int _{\caliD} P_{n}(z)  \overline{z}^ndA$, by \eqref{extremality}.

To accomplish this, we consider the meromorphic kernel 
\begin{align}\label{defkernelM}
\mathcal{M}_{\mathcal{D}}(z,\zeta) 
&:=  \sum_{\tau \in \mathcal{T}}
\frac{\tau(z)-\tau(0)}{(\zeta - \tau(0))(\zeta - \tau(z))}
,
\end{align}
which is related to the reproducing kernel $\mathcal{K}_{\mathcal{D}}(z,\zeta)$ via the equality 
\begin{align}\label{twokernelsrelation}
\frac{\partial}{\partial z} \mathcal{M}_{\mathcal{D}} (z, \zeta) =
\frac{ \mathcal{K}_{\mathcal{D}}(z,1/\cj{\zeta})}{ \zeta^2 }.
\end{align}
We then fix $r\in (1,\rho_a^{-1})$ and note that 
\begin{align}\label{functionsfrhorel-3} 
\mathcal{P}_{n} (z)=\frac{1}{2 \pi i } \ointctrclockwise\limits_{\T(0,r)}f_{n,e} (\zeta)  \mathcal{M}_{\mathcal{D}} (z, \zeta)  \zeta^{n+1}  d \zeta,\quad |z|<r.
\end{align}
Indeed, being  $\mathcal{M}_{\mathcal{D}}$ given by \eqref{defkernelM}, the right-hand side of \eqref{functionsfrhorel-3} is well-defined at those points $z \in D(0,\rho_a^{-1})\setminus \T(0,r)$ for which
$\tau(z)\not\in \T(0,r)$ for every $\tau\in \caliT^*$. But we know from \eqref{ineq:mandM1}-\eqref{ineq:mandaM3} in Lemma \ref{lemma1} that for all $\tau\in \caliT^*$, $\tau(D(0,r))\subset D(0,r)$, so that the integral in \eqref{functionsfrhorel-3} is well-defined and analytic for  $z\in D(0,r)$. Moreover, using \eqref{defkernelM}, the residue theorem, and the definition \eqref{f2k+1cont2}, we get that for $z\in D(0,r)$, 
\begin{align*}
 &\frac{1}{2 \pi i } \ointctrclockwise\limits_{\T(0,r)}f_{n,e} (\zeta)  \mathcal{M}_{\mathcal{D}} (z, \zeta)  \zeta^{n+1}  d \zeta\\
 &={}
 \sum _{k=0} ^\infty \frac{1}{2 \pi i } \ointctrclockwise\limits_{\T(0,r)}f_{n,2k} (\zeta)  \mathcal{M}_{\mathcal{D}} (z, \zeta)  \zeta^{n+1}  d \zeta\\
 &={}\sum _{k=0} ^\infty\frac{z}{2 \pi i } \ointctrclockwise\limits_{\T(0,r)}\frac{f_{n, 2k} (\zeta)\zeta^{n} }
{\zeta - z} d \zeta\\
&\qquad + \sum _{k=0} ^\infty\sum_{\tau \in \caliT^*}\frac{\tau(z)-\tau(0)}{2 \pi i }\ointctrclockwise\limits_{\T(0,r)} \frac{f_{n, 2k} (\zeta)\zeta^{n+1}d \zeta}{(\zeta - \tau(0))(\zeta - \tau(z))} \\
&={} z^{n+1}\sum _{k=0} ^\infty f_{n, 2k} (z)+\sum _{k=0}^\infty f_{n,2k+1}(z),
\end{align*}
which is precisely the value of $\mathcal{P}_{n} (z)$.

By \eqref{functionsfrhorel-3} and  \eqref{twokernelsrelation},
\[
\mathcal{P}_n '(z)= 
\frac{1}{2 \pi i}
\ointctrclockwise\limits_{\T(0,r)}
f_{n,e}(\zeta)\mathcal{K}_{\mathcal{D}}(z,1/\cj{\zeta})\zeta^{n-1} d \zeta,
\quad
|z|<r.
\]
By Fubini's theorem and the reproducing property of the kernel $\mathcal{K}_\caliD$, we then have 
\begin{align*}
\int _{\caliD} \mathcal{P}'_{n} (z)  \overline{z}^m  d A(z)={}&\frac{1}{2\pi i}\int \limits_{\T(0,r)}f_{n,e} (\zeta) \zeta^{n-1} 
\left\{ \overline{
\int_{\caliD}   \mathcal{K}_{\caliD}(1/\cj{\zeta},z)  z^m  dA(z) }
\right\}d \zeta\\
={} &\frac{1}{2\pi i}\int \limits_{\T(0,r)}f_{n,e} (\zeta) \zeta^{n-m-1} 
d \zeta=\begin{cases}
    0,& 0 \leq m< n,\\
f_{n,e} (0), & m=n.
\end{cases}
\end{align*}
\end{proof}
\section{Proofs of the asymptotic results}
Before commencing the proofs of the asymptotic results, we gather in three auxiliary  propositions the finer aspects of the asymptotic analysis. Because the proofs of these propositions are technically involved, we postpone them 
to the last section of the paper.
\subsection{Auxiliary propositions}
For every $j\in \N_s$, let $x_j$ and $y_j$ be defined as in \eqref{def-of-xj-yj}

 \begin{proposition} \label{prop-integral-error-1} For every $j\in \N_s$, 
the asymptotic expansion
\begin{align}\label{asymptotic-formula-1-in-z}
\begin{split}
&\frac{1}{2 \pi i}\!\!\! \ointctrclockwise\limits_{\T(0,|x_j|)}
\frac{T'_j(\zeta)\left(T_j(\zeta)/\zeta\right)^{n}}{\zeta-z}  d\zeta\\
&\sim{} -\frac{|x_j|^{-2n-2}}{2\pi}\sum_{k=0}^\infty R_{j,k}(z)\frac{\Gamma(k+\frac{1}{2})\Gamma(n-k+\frac{3}{2})}{\Gamma(n+2)}
\end{split}
\end{align}
holds true uniformly on closed subsets of $\cj{\C}\setminus\T(0,|x_j|)$  as $n\to\infty$, where the coefficients $R_{j,k}$ are defined via the Maclaurin series \eqref{coefficients-Rjk}.  

Also, when $r_j=|c_j|$, we have 
\begin{align}\label{particular-case}
\frac{1}{2 \pi i} \ointctrclockwise\limits_{\T(0,|x_j|)} 
\frac{T'_j(\zeta)\left(T_j(\zeta)/\zeta\right)^{n}}{\zeta}  d\zeta
={} & \frac{|x_j|^{-2n-2}}{2\sqrt{\pi} } \frac{\Gamma(n+3/2)}{\Gamma(n+2)}.
\end{align}
 \end{proposition} 
 \begin{proposition} \label{prop-integral-error-2}
For every $j\in \N_s$, there exists a constant $M_j$ such that for every integer $n\geq 0$ and $z\in \T(0,|x_j|)$,
 \begin{align}\label{BigO-formula-1}
\left|\frac{1}{2 \pi i} \ointctrclockwise\limits_{\T(0,|x_j|)} 
T'_j(\zeta)\frac{\left(T_j(\zeta)/\zeta\right)^{n}-\left(T_j(z)/z\right)^{n}}{\zeta-z}  d\zeta\right|
\leq M_j|x_j|^{-2n}
\end{align}
Moreover, 
\begin{align}\label{BigO-formula-1iminf}
\lim_{n\to\infty}\frac{|x_j|^{2(n+2)}}{2 \pi i} \ointctrclockwise\limits_{\T(0,|x_j|)} 
T'_j(\zeta)\frac{\left(T_j(\zeta)/\zeta\right)^{n}-\left(T_j(x_j)/x_j\right)^{n}}{\zeta-x_j}  d\zeta=-1/2.
\end{align}
 \end{proposition} 
 
 For the next proposition, recall that the functions $\Theta_\sigma$ have been introduced in \eqref{definition-theta}.
\begin{proposition}\label{cdcd}
If $j \in \N_s$ is such that $|a_j|>0$,  then  
\begin{align}\label{series-behavior}
\sum_{v=1}^\infty   (T_j ^v(z))^n (T_j ^v)' (z) ={} &\frac{a_j^{n+1}}{n}\frac{\Phi'_j(z)}{\Phi_j(z)}\Theta_{\sigma_j^2}(n\alpha_j\Phi_j(z))+ O(|a_j|^n/n^2)
\end{align}
uniformly on closed subsets of  $\cj{D(0,|a_j|)}\setminus\{a_j\}$ as $n \to \infty$.
\end{proposition}

\subsection{Proof of Theorem \ref{thm-asymp-leading}}
By Theorem \ref{thm-poly},  we can write 
 \begin{align}\label{series-for-kappa_n}
(n+1) \kappa_n ^{-2} = 1+f_{n, 2} (0)+\sum_{k=2}^\infty f_{n, 2k} (0).
\end{align}
Using \eqref{ineq:first-1} with $r=\rho_x$, we find  that  for all $n$ large,
\begin{align}\label{tail-estimate}
\sum_{k=2}^\infty |f_{n, 2k} (0)|\leq {} &\frac{(\rho_x-m(\rho_x))V(\rho_x,n)^2}{\rho_x(1-V(\rho_x,n))}=O((m(\rho_x)/\rho_x)^{2n}).
\end{align}
From the definition \eqref{f2k+1cont2} and Lemma \ref{estimate-sum-over-tau_v}, we have 
\begin{align*}
f_{n,1}(z)= &\sum_{j=1}^s T_j(z)^{n+1}- \sum_{j=1}^sT_j(0)^{n+1}+O((m^2(\rho_x))^n)
\end{align*}
uniformly for $z\in \cj{D(0,\rho_x)}$ as $n\to\infty$. This and the definition \eqref{functionsf1} readily yield
\begin{align*}
f_{n,2} (0)={}&-\sum_{j=1}^s  \frac{1}{2 \pi i} \ointctrclockwise\limits_{\T(0,\rho_x)} \frac{T_j(\zeta)^{n+1}}{\zeta^{n+2}} d\zeta +O((m^2(\rho_x)/\rho_x)^n).
\end{align*}
Inserting this estimate and  that of \eqref{tail-estimate} in  \eqref{series-for-kappa_n} we find
 \begin{align}\label{baserelation}
 \begin{split}
(n+1) \kappa_n ^{-2} ={} & 1-\sum_{j=1}^s  \frac{1}{2 \pi i} \ointctrclockwise\limits_{\T(0,\rho_x)} \frac{T_j(\zeta)^{n+1}}{\zeta^{n+2}} d\zeta+O(\beta_x^{n}),
\end{split}
\end{align}
with 
\[
\beta_x=\max\left\{\frac{m^2(\rho_x)}{\rho_x},\frac{m(\rho_x)^2}{\rho_x^2}\right\}<\rho_x^{-1}.
\]

Cauchy's theorem and integration by parts yield 
\begin{align*}
 \ointctrclockwise\limits_{\T(0,\rho_x)} \frac{T_j(\zeta)^{n+1}}{\zeta^{n+2}} d\zeta ={} & \ointctrclockwise\limits_{\T(0,|x_j|)}\frac{T_j(\zeta)^{n+1}}{\zeta^{n+2}} d\zeta=\ointctrclockwise\limits_{\T(0,|x_j|)} 
\frac{T'_j(\zeta)\left(T_j(\zeta)/\zeta\right)^{n}}{\zeta}  d\zeta,
\end{align*}
which together with \eqref{baserelation} and Proposition \ref{prop-integral-error-1} (with $z=0$) readily prove both Theorem \ref{thm-asymp-leading} and \eqref{degenerate-case}.

\subsection{Proof of Theorem \ref{thm-cmcds}} We first prove \eqref{expansion-exterior-asymptotics} and the validity of \eqref{exterior-asymptotics} for $r\geq \rho_x$. Let $\rho\in (\rho_a,\rho_a^{-1})$. During the proof of Theorem \ref{thm-poly}, we obtained (see \eqref{P_n-for-z-large})
 \begin{align}\label{other-important-identity-2}
 \begin{split}
\frac{P_n(z)}{z^n}-1={} & 
 -\frac{1}{2 \pi i} \ointctrclockwise\limits_{\T(0,\rho)}f_{n,o}(\zeta) \frac{\zeta^{-n-1}}{\zeta - z }  d\zeta\\
 &-\frac{z}{(n+1)2 \pi i} \ointctrclockwise\limits_{\T(0,\rho) } f_{n,o}(\zeta) \frac{\zeta^{-n-1}}{(\zeta - z)^2 }  d \zeta
 \end{split}
 \end{align}
 for all $z\in\Delta(0,\rho)$ and $n$ larger than some ($\rho$-dependent) number. Here and as previously introduced,
 \[
   f_{n,o}(z)=\sum _{k=1} ^\infty f_{n, 2k-1} (z).
 \]

 Just as we derived \eqref{other-important-identity} (see the paragraph succeeding that identity), we can use integration by parts in \eqref{other-important-identity-2} to get
  \begin{align}\label{ineq-vrox-0}
\frac{P_n(z)}{z^n}-1= -\frac{1}{(n+1)2 \pi i} \ointctrclockwise\limits_{\T(0,\rho) } \frac{ f'_{n,o}(\zeta)\zeta^{-n}}{\zeta - z }  d \zeta,\qquad  z\in\Delta(0,\rho).
 \end{align}
Since
\[
f'_{n,1}(z)=(n+1)\sum_{j=1}^sT_j'(z)T_j(z)^n+(n+1)\sum_{\tau:\ell(\tau)\geq 2} \tau'(z)\tau(z)^n,
\]
we obtain from \eqref{estimate-derivative-f_n,1}, \eqref{ineq:first-2}, and \eqref{f2k+1cont2} that
 \begin{align}\label{ineq-vrox}
 \begin{split}
 -\frac{1}{(n+1)2 \pi i} \ointctrclockwise\limits_{\T(0,\rho) }\frac{ f'_{n,o}(\zeta)\zeta^{-n}}{\zeta - z }  d \zeta= {} &-\sum_{j=1}^s\frac{1}{2 \pi i} \ointctrclockwise\limits_{\T(0,\rho) } \frac{ T'_j(\zeta)(T_j(\zeta)/\zeta)^{n}}{\zeta - z }  d \zeta\\
& + O((m^2(\rho)/\rho)^n)
 +O((m(\rho)/\rho)^{2n})
\end{split}
 \end{align} 
 uniformly on closed subsets of $\cj{\C}\setminus \T(0,\rho)$ as $n\to\infty$. 

Let us  momentarily  set
\begin{align*}
v(\rho):= {} &\max\{|T_j(z)/z|:|z|=\rho, \ j\in \N_s, \ |x_j|>\rho_x\},\\
\tilde{v}(\rho):={} & \max\{v(\rho), m^2(\rho)/\rho,(m(\rho)/\rho)^{2}\}.
\end{align*}
Combining  \eqref{ineq-vrox-0} and \eqref{ineq-vrox}, we obtain 
\begin{align}\label{ineq-vrox-1}
 \begin{split}
\frac{P_n(z)}{z^n}-1= {} &-\sum_{j:|x_j|=\rho_x}^s\frac{1}{2 \pi i} \ointctrclockwise\limits_{\T(0,\rho) }\frac{ T'_j(\zeta)(T_j(\zeta)/\zeta)^{n}}{\zeta - z }  d \zeta+ O(\tilde{v}(\rho)^n) 
\end{split}
 \end{align} 
 uniformly on closed subsets of $\Delta(0,\rho)$ as $n\to\infty$.  

In terms of the functions $m_j$ introduced in \eqref{defmj},
\begin{align*}
v(\rho)= {} &\max\{m_j(\rho)/\rho: \ j\in \N_s, \ |x_j|>\rho_x\}.
\end{align*}
By \eqref{vrhox}, $v(\rho_x)<\rho_x^{-2}$, so that $
\tilde{v}(\rho_x)<\rho_x^{-2}$ as well. Then,  choosing  $\rho=\rho_x$  in \eqref{ineq-vrox-1} and replacing the integrals in  \eqref{ineq-vrox-1} by the  expansions  \eqref{asymptotic-formula-1-in-z}  quickly yields \eqref{expansion-exterior-asymptotics}. 

For $r>\rho_x$, the equality \eqref{exterior-asymptotics} follows directly from \eqref{expansion-exterior-asymptotics}. To prove \eqref{exterior-asymptotics} for $r=\rho_x$, we choose $\rho<\rho_x$ so closed to $\rho_x$ that
\begin{align}\label{tildevro-estimate}
\tilde{v}(\rho)<\rho_x^{-2}.
\end{align} 
Since for every $z\in \T(0,\rho_x)$, the function (in the variable $\zeta$) $T_j'(\zeta)/(\zeta-z)$ is analytic in $D(0,\rho_x)$, we get by an application of  Cauchy's integral theorem that for all $z\in \T(0,\rho_x)$,
\begin{align*} 
\ointctrclockwise\limits_{\T(0,\rho) } \frac{ T'_j(\zeta)(T_j(\zeta)/\zeta)^{n}}{\zeta - z }  d \zeta={} &\ointctrclockwise\limits_{\T(0,\rho) } T'_j(\zeta)\frac{ (T_j(\zeta)/\zeta)^{n}-(T_j(z)/z)^{n}}{\zeta - z }  d \zeta\\
={} &\ointctrclockwise\limits_{\T(0,\rho_x) } T'_j(\zeta)\frac{ (T_j(\zeta)/\zeta)^{n}-(T_j(z)/z)^{n}}{\zeta - z }  d \zeta.
 \end{align*} 
Because of this equality, it follows from \eqref{ineq-vrox-1}, \eqref{BigO-formula-1} and  \eqref{tildevro-estimate} that \eqref{exterior-asymptotics} holds true for $r=\rho_x$.
 
We now prove \eqref{cor-eqn-one}, \eqref{cor-eqn-two}, and the validity of \eqref{exterior-asymptotics} for $\rho_a<r<\rho_x$. 

Once $n$ is sufficiently large,  the inequality \eqref{induction-case-1} holds true with $v=1$, and so we have
\[
\sum _{k=0} ^\infty\sum_{\tau \in \caliT^*}\left|\left(f_{n,2k} (\tau(z)) \tau(z)^{n+1}\right)'\right|\leq \frac{(n+1)\mu(\rho_x)m(\rho_x)^n}{1-V(\rho_x,n)},\quad |z|\leq \rho_x,
\] 
It then follows  from  \eqref{f2k+1cont2}  and Theorem \ref{thm-poly} that for all  $|z|<\rho_x$,
\begin{align*}
P_n (z)= {} &(n+1)^{-1}\sum_{\tau \in \caliT}\sum _{k=0} ^\infty\left(f_{n,2k} (\tau(z)) \tau(z)^{n+1}\right)'\\
={} &\sum_{\tau \in \caliT}\tau(z)^{n}\tau'(z)\left(1+
K_n (\tau(z)\right),
\end{align*} 
with
\begin{align*}
K_n(z)={} & \sum _{k=1} ^\infty f_{n,2k} (z)+\frac{z}{n+1}\sum _{k=1} ^\infty f'_{n,2k} (z),\quad |z|<\rho_x.
\end{align*}

According to \eqref{other-important-identity}, 
 \begin{align*}
\begin{split}
K_n(z)={} & - \frac{(n+1)^{-1}}{2 \pi i} \ointctrclockwise\limits_{\T(0,\rho_x) } \frac{ (\sum _{k=1} ^\infty f'_{n,2k-1}(\zeta))\zeta^{-n}}{\zeta - z }  d \zeta,
\end{split}
\end{align*}
which, together with \eqref{ineq-vrox}  for $\rho=\rho_x$ and the expansion \eqref{asymptotic-formula-1-in-z}, yields  \eqref{estimate-K_n}.

By \eqref{cor-eqn-one}-\eqref{estimate-K_n}, we have
\begin{align}\label{formula-for-Pn-inside-circle-rox}
\begin{split}
P_n(z)= {} &z^n\left(1+O\left(n^{-1/2}\rho_x^{-2n}\right)\right)+\sum_{\tau \in \caliT^*}\tau(z)^{n}\tau'(z)\\
& +O\left(n^{-1/2}\rho_x^{-2n}\sum_{\tau \in \caliT^*}|\tau(z)|^{n}|\tau'(z)|\right)
\end{split}
\end{align}
locally uniformly in $D(0,\rho_x)$ as $n\to\infty$. 

It follows from \eqref{formula-for-Pn-inside-circle-rox}, \eqref{estimate-derivative-f_n,1}, \eqref{defmandMr}, and \eqref{maxproperty}, that for every $r\in (\rho_a,\rho_x)$,
\begin{align*}
\frac{P_n (z)}{z^n}-1= {} & \sum_{j=1}^sT_j'(z)\left(\frac{T_j(z)}{z}\right)^n+ O\left(n^{-1/2}\rho_x^{-2n}\right)+O((m^2(r)/r)^n)\\
={} & O((m(r)/r)^n)
\end{align*}
uniformly on $\T(0,r)$ as $n\to\infty$.

We finish then with the  proof of  \eqref{cor-eqn-two}. Since  $T_j(D(0,\rho_x)) \subset D(0,\rho_x)$ for each $j \in \N_s$, we can evaluate equation \eqref{cor-eqn-one} at $T_j$ to get
\begin{equation}
\label{p-n-t-j}
P_n(T_j(z)) =  \sum_{\tau \in \mathcal{T}} ((\tau \circ T_j)(z))^n  \tau'(T_j(z))  ( 1 + (K_n  \circ \tau \circ T_j) (z)), \quad |z|<\rho_x.
\end{equation}
The only transformation in $\mathcal{T} \setminus \mathcal{T}^*$ is the identity function, and  $\mathcal{T}^* = \bigcup_{j=1} ^s \{ \tau T_j : \tau \in \mathcal{T}\}$, this being a disjoint union. Hence, equation \eqref{cor-eqn-one} can be written as   
\begin{align*}
P_n(z)= {}& z^n(1+ K_n(z)) +  \sum_{\tau \in \caliT^*}\tau (z)^n  \tau '(z)  ( 1 + (K_n  \circ \tau) (z))\\
={} &z^n (1+ K_n(z)) \\
&+ \sum_{j=1}^s \sum_{\tau \in \mathcal{T}}((\tau \circ T_j )(z))^n (\tau \circ T_j) '(z) ( 1 + (K_n \circ  \tau \circ T_j) (z)), 
 \end{align*}
which is another way to write \eqref{cor-eqn-two}, owing to  (\ref{p-n-t-j}). 
\subsection{Proof of Theorem \ref{thm-cmcds-2}}
According to \eqref{formula-for-Pn-inside-circle-rox} and \eqref{estimate-derivative-f_n,1} (recall that $m(\rho_a)=\rho_a$)
\begin{align*}
P_n(z)= {} &z^n+\sum_{\tau \in \caliT^*}\tau(z)^{n}\tau'(z)+O\left(n^{-1/2}(\rho_a/\rho_x^{2})^{n}\right)
\end{align*}
uniformly in $z\in\cj{D(0,\rho_a)}$ as $n\to\infty$.

The family $\mathcal{T}_j$ was defined in \eqref{defTj} as the set of all transformations $\tau$ whose terminal operator is $T_j$.  Therefore, we can write
\begin{align*}
\sum_{\tau \in \caliT^*}\tau(z)^{n}\tau'(z)=\sum_{j:|a_j|=\rho_a}\sum_{\tau \in \caliT_j}\tau(z)^{n}\tau'(z)+\sum_{j:|a_j|<\rho_a}\sum_{\tau \in \caliT_j}\tau(z)^{n}\tau'(z).
\end{align*}

From \eqref{ineq:mandaM3}, we know that for every $\tau\in \caliT$, $\tau(\cj{D(0,\rho_a)})\subset \cj{D(0,\rho_a)}$, and because of \eqref{between},  
\[
\varrho_a:=\max\{|T_j(z)|:z\in \cj{D(0,\rho_a)},\ j\in \N_s,\ |a_j|<\rho_a \} <\rho_a,
\]
so that $\tau(\cj{D(0,\rho_a)})\subset \cj{D(0,\varrho_a)}$ for all $\tau\in \caliT_j$ with $|a_j|<\rho_a$. Hence,
\begin{align}\label{asympt-preliminary}
\begin{split}
P_n(z)= {} &z^n+\sum_{j:|a_j|=\rho_a}\sum_{\tau \in \caliT_j}\tau(z)^{n}\tau'(z)+O(\varrho_a^n)+O\left(n^{-1/2}(\rho_a/\rho_x^{2})^{n}\right)
\end{split}
\end{align}
uniformly on $\cj{D(0,\rho_a)}$ as $n\to\infty$. 

The set $\mathcal{T} \setminus \mathcal{T}_j$ is the collection of all transformations with a terminal operator different from $T_j$, together with the identity transformation $T_0$, and note that
\[
\mathcal{T}_j = \bigcup _{v=1}^\infty  \{ T_j^v \tau: \tau \in \mathcal{T} \setminus \mathcal{T}_j \} .
\] 
Therefore,  we can write 
\begin{align}\label{sum-tau-not-in-Tj}
\sum_{\tau \in \mathcal{T}_j } \tau(z)^n\tau'(z) &= 
\sum_{\tau \in \mathcal{T} \setminus \mathcal{T}_j} 
\tau'(z)\sum_{v=1} ^\infty  
T_j ^v(\tau(z))^n (T_j ^v)'(\tau(z)).
\end{align}

Let $\epsilon>0$ be such that $D(a_j,\epsilon)\subset D(c_j,r_j)$ whenever $|a_j|=\rho_a$, and let us set 
\[
E_j:=\cj{D(0,\rho_a)}\setminus D(a_j,\epsilon),\qquad E_\epsilon=\bigcap_{j:|a_j|=\rho_a}E_j.
\]

If $\tau\not\in \mathcal{T}_j$, then either $\tau$ is the identity (in which case $\tau(E_\epsilon)\subset E_j$), or $\tau=T_k\tau_1$ for some $k\not=j$ and some $\tau_1\in \mathcal{T}$, so that by \eqref{inclusion}, $\tau(\cj{D(0,\rho_a)})\subset D(c_k,r_k)$. Since $D(c_k,r_k)\cap  D(c_j,r_j)=\emptyset$, we conclude that $\tau(\cj{D(0,\rho_a)})\subset E_j$. Summarizing, we have found that $\tau(E_\epsilon)\subset E_j$ for all $\tau\not\in \mathcal{T}_j$, which allows us to apply Proposition \ref{cdcd} to the inner sum in the right-hand side of \eqref{sum-tau-not-in-Tj} and get
\begin{align*}
\sum_{\tau \in \mathcal{T}_j } \tau(z)^n  \tau'(z) &= \frac{a_j^{n+1}}{n}
\sum_{\tau \in \mathcal{T} \setminus \mathcal{T}_j} \!\!
\tau'(z)\left(\frac{\Phi'_j(\tau(z))}{\Phi_j(\tau(z))}\Theta_{\sigma_j^2}(n\alpha_j\Phi_j(\tau(z)))+O(n^{-1})\right)
\end{align*}
uniformly in $z\in E_\epsilon$ as $n\to\infty$. Inserting this formula into \eqref{asympt-preliminary} quickly yields \eqref{behavior-inside-Drhoa}, since every closed subset of $\cj{D(0,\rho_a)}\setminus \{a_j:|a_j|=\rho_a\}$ is contained in some $E_\epsilon$ with sufficiently small $\epsilon$.

To prove \eqref{behavior-at-points-aj}, we evaluate \eqref{cor-eqn-two} at $t=a_j$ ($j$ such that $|a_j|=\rho_a$) to get 
\begin{equation*}
P_n(a_j) = a_j^n+ 
P_n(a_j)  T_j '(a_j)+ \sum_{k\in\N_s\setminus\{j\}} P_n(T_k(a_j))  T_k '(a_j)+O\left(\frac{1}{\sqrt{n}}\left(\frac{\rho_a}{\rho_x^{2}}\right)^{n}\right).
 \end{equation*}

From \eqref{behavior-inside-Drhoa}, we know that  $P_n(z)=O(\rho_a^n/n)$ for every $z\in D(0,\rho_a)$, and if $k\not=j$, then $T_k(a_j)\in D(0,\rho_a)$, so that, as $n\to \infty$, 
\begin{equation*}
P_n(a_j) = a_j^n+ P_n(a_j)  T_j '(a_j)+ O(\rho_a^n/n),
\end{equation*}
which turns into \eqref{behavior-at-points-aj} after computing $T_j '(a_j)$.

\section{Proofs of the auxiliary propositions}

\subsection{Proof of Proposition \ref{prop-integral-error-1}}

The points $x_j$ and $y_j$, as defined by \eqref{def-of-xj-yj}, are the reflections about the unit circle of the two points where the circle $\T(c_j,r_j)$ intersects the line that passes through $0$ and $c_j$. Manipulating \eqref{formula-Tj-1} we can get to express $T_j$ in terms of $x_j$ and $y_j$,  resulting in 
\begin{align}\label{formula-Tj}
T_j(z)=\frac{x_j}{\cj{x}_j}\frac{\frac{y_j+x_j}{2}-z}{y_jx_j-z\frac{y_j+x_j}{2}}.
\end{align}

Using \eqref{formula-Tj} we can compute the derivative of $T_j(z)/z$ and see that    $x_j$ and $y_j$  are the critical points of $T_j(z)/z$. It can be easily verified that 
\[
1<|x_j|<|a_j|^{-1}<|p_j|, \quad |x_j|<|y_j|,\quad 
T_j(x_j)=\cj{x}_j^{-1},
\quad T_j(y_j)=\cj{y}_j^{-1},
\]
where 
\begin{align}\label{pj}
p_j=\frac{2x_jy_j}{y_j+x_j}
\end{align}
is the pole of $T_j(z)$. 

Let us use $I_{n}(z)$ to denote the integral in the left-hand side of  \eqref{asymptotic-formula-1-in-z}, that is, 
\begin{align}\label{def-In}
I_n(z):=\frac{1}{2 \pi i} \ointctrclockwise\limits_{\T(0,|x_j|) } 
\frac{T'_j(\zeta)\left(T_j(\zeta)/\zeta\right)^{n}}{\zeta-z}  d\zeta,\quad |z|\not=|x_j|.
\end{align}

Let $C_j$ be the circle that passes through $x_j$ and $y_j$ and is symmetric  about the line $\ell_{j}:=\{tc_j:t\in\R\}$. When  $r_j=|c_j|$, $C_j$ is actually the line perpendicular to $\ell_j$ at $x_j$. When $r_j\not=|c_j|$, we will think of $C_j$ as a positively oriented contour, and when $r_j=|c_j|$, the orientation of $C_j$ will be in the direction of the vector $ix_j$.

Let $E$ be a closed subset of $\cj{\C}\setminus\T(0,|x_j|)$. We can find   a positive number
$\varepsilon$, smaller than the  distance between $E$ and $\T(0,|x_j|)$ and such that $T_j(z)/z$ is analytic in the annulus $|x_j|-\varepsilon<|z|<|x_j|+\varepsilon$. It is also possible to find  $r\in (|x_j|-\varepsilon,|x_j|+\varepsilon)$ such that the circle  $\T(0,r)$ intersects $C_j$ at two distinct points. Indeed, if $r_j>|c_j|$, then any $r \in (|x_j|-\varepsilon,|x_j|)$ will do, while if $r_j\leq |c_j|$, then  $r$ needs to be chosen greater than $|x_j|$.

The two points at which $\T(0,r)$ intersects $C_j$  are the end points of two arcs  of $C_j$. Of these two arcs, let us denote by $C_{j,r}$ the one  containing the point $x_j$. Let $\T_{j,r}$ denote the arc of $\T(0,r)$ that falls into the closure of the unbounded component of $\C\setminus C_j$ (if $r_j=|c_j|$ we choose $\T_{j,r}$ lying to the left of the line $C_j$), and let $L_{j,r}:=\T_{j,r}\cup C_{j,r}$, a Jordan contour which we consider to be positively oriented.  By Cauchy's theorem, we have   that for all $z\in E$,     
\begin{align}\label{I_n-formula-1-z}
\begin{split}
I_{n}(z)={} &  \frac{1}{2 \pi i}\ointctrclockwise_{L_{j,r}} T'_j(\zeta)\frac{(T_j(\zeta)/\zeta)^{n}}{\zeta - z}d\zeta\\
={} &\frac{1}{2 \pi i}\int_{\T_{j,r}} T'_j(\zeta)\frac{(T_j(\zeta)/\zeta)^{n}}{\zeta - z}d\zeta+\frac{\epsilon_j}{2 \pi i}\int_{C_{r,j}} T'_j(\zeta)\frac{(T_j(\zeta)/\zeta)^{n}}{\zeta - z}d\zeta.
\end{split}
\end{align}

In the latter two integrals, the orientation of $\T_{j,r}$ and $C_{j,r}$ is the one that they  inherit as arcs of the positively oriented circles $\T(0,r)$ and $C_j$, respectively. When $r_j<|c_j|$, the orientation that $C_{j,r}$ inherits as an arc of $C_j$ is the opposite it inherits as an arc of the curve $L_{j,r}$, hence the need for the factor $\epsilon_j$  defined  in \eqref{e_j-def}. 

From the explicit representations for $z/T_j(z)$ that we give below, it is easy to see that $z/T_j(z)$ is a conformal map of any of the two components of $\cj{\C}\setminus C_j$  onto the exterior of the segment $[|x_j|^{2},|y_j|^{2}]$, mapping $\infty$ to $\infty$, and that   
\begin{align*}
q_{j,r}:=\max_{\zeta\in \T_{j,r}}|T_j(\zeta)/\zeta|<|T_j(x_j)/x_j|=|x_j|^{-2}.
\end{align*} 
This and \eqref{I_n-formula-1-z} give 
\begin{align}\label{I_n-formula-2-z}
\begin{split}
I_{n}(z)
={} &\frac{\epsilon_j}{2 \pi i}\int_{C_{r,j}} T'_j(\zeta)\frac{(T_j(\zeta)/\zeta)^{n}}{\zeta - z}d\zeta+O(q_{j,r}^n), \quad z\in E.
\end{split}
\end{align}

Let $\Omega_j$ denote the unbounded component of $\cj{\C}\setminus C_j$ when $r_j\not=|c_j|$, and let it denote the semi-plane lying to the right of the line $C_j$ when $r_j=|c_j|$.  Let 
\[
g:\cj{\C}\setminus [|x_j|^{2},|y_j|^{2}]\to  \Omega_j
\] 
denote the inverse of $z/T_j(z)$, and let $\lambda_j^{2}$ be the point of $(|x_j|^2,|y_j|^2)$ which is the image by $z/T_j(z)$ of the endpoints of $C_{j,r}$.

If $t\in \R$ and $f$ is a function defined at all non-real points of some neighborhood of $t$, then we will use $f_+(t)$ and $f_-(t)$  to respectively denote the limit of $f(z)$ as $z$ approaches $t$ from the upper and lower half-planes.   

We can then make the change of variable $\zeta=g_\pm(t)$ to express the integral over $C_{j,r}$ in \eqref{I_n-formula-2-z} as the sum of two integrals over the interval $[|x_j|^2,\lambda_j^2]$ to  arrive at 
\begin{align}\label{I_n-formula-3-z}
\begin{split}
I_{n}(z)
={} &\frac{-\epsilon_j}{2\pi i } \int_{|x_j|^2}^{\lambda_j^2} \left(\frac{T'_j(g_+(t))g'_+(t)}{g_+(t)-t} -\frac{T'_j(g_-(t))g'_-(t)}{g_-(t)-t} \right)t^{-n}dt+O(q_{j,r}^n).
\end{split}
\end{align}

We now need to rely on explicit computations. We first assume that $r_j\not=|c_j|$. In this case, we can write the mapping $T_j(z)/z$  as the composition 
\[
 \frac{z}{T_j(z)}=\frac{\cj{x}_j(y_j^2-x_j^2)}{2x_j}\left(\frac{w+w^{-1}}{2}\right)+\frac{\cj{x}_j(y_j^2+x_j^2)}{2x_j},
\]
with 
\[
 w=\frac{2}{y_j-x_j}\left(z-\frac{y_j+x_j}{2}\right).
\]

Hence, $z/T_j(z)$ behaves essentially as the Zhukovsky transformation $(w+w^{-1})/2$, so that $z/T_j(z)$ is a conformal map of the exterior (and of the interior) of the circle $C_j$  onto the exterior of the segment $[|x_j|^{2},|y_j|^{2}]$, mapping $\infty$ to $\infty$.

The inverse $g(t)$ of $z/T_j(z)$ is then given by
\begin{align}\label{exp-for-g-case1}
g(t)=\frac{y_j+x_j}{2}+\frac{y_j-x_j}{2}J\left(\frac{2x_jt}{\cj{x}_j(y_j^2-x_j^2)}-\frac{y_j^2+x_j^2}{y_j^2-x_j^2}\right),
\end{align}
where $J(u)=u+\sqrt{u^2-1}$ is the inverse of the Zhukovsky transformation. Here the branch of the square root in $\C\setminus[-1,1]$ is chosen so as to have $\sqrt{u^2-1}>0$ for $u>1$.

To keep the upcoming expressions as clean as possible, let us set 
\begin{align}\label{def-u-and-A}
u=\frac{2x_jt}{\cj{x}_j(y_j^2-x_j^2)}-\frac{y_j^2+x_j^2}{y_j^2-x_j^2}, \quad A=\frac{y_j+x_j}{y_j-x_j},
\end{align}
so that 
\begin{align*}
g(t)=\frac{y_j-x_j}{2}(J(u)+A),
\end{align*}
and consequently
\begin{align*}
y_jx_j-g(t)\frac{y_j+x_j}{2}=-\left(\frac{y_j-x_j}{2}\right)^2(1+AJ(u)).
\end{align*}

 Differentiating  \eqref{exp-for-g-case1} gives
\begin{align*}
g'(t)=\frac{x_j}{\cj{x}_j}\frac{1}{y_j+x_j}\frac{J\left(u\right)}{\sqrt{
u^2-1
}},
\end{align*}
and since 
\[
T_j'(z)=\frac{x_j}{\cj{x}_j}\frac{\left(\frac{y_j-x_j}{2}\right)^2}{\left(y_jx_j-z\frac{y_j+x_j}{2}\right)^2},
\]
we have
\begin{align}\label{needed-for-int}
\frac{T_j'(g(t))g'(t)}{g(t)-z}={} & \frac{H(t)}{\sqrt{(t-|x_j|^2)(t-|y_j|^2),
}}
\end{align}
where
\[
H(t)= \frac{\frac{x_j}{\cj{x}_j}\left(\frac{2}{y_j-x_j}\right)^2J\left(u\right)}{(1+AJ(u))^2\left(J(u)+A-\frac{2z}{y_j-x_j}\right)}.
\]

Using that $\pm i\sqrt{(t-|x_j|^2)(|y_j|^2-t)}$ are, respectively, the $\pm$ boundary values at $t\in (|x_j|^2,|y_j|^2)$ of the function $\sqrt{(z-|x_j|^2)(z-|y_j|^2)}$,  we get from \eqref{needed-for-int} and \eqref{I_n-formula-3-z} that uniformly in $z\in E$ as $n\to\infty$,
\begin{align}\label{Inexpress-1}
\begin{split}
I_{n}(z)
={} & \frac{\epsilon_j|x_j|^{-2n}}{2\pi } \int_{1}^{\left|\frac{\lambda_j}{x_j}\right|^2} \frac{(H_+(|x_j|^2x)+H_-(|x_j|^2x))x^{-n}}{\sqrt{(y_j/x_j)^2-x}\sqrt{x-1}}dx+O(q_{j,r}^n).
\end{split}
\end{align}

With the help of the relations $J_+(u)J_-(u)=1$ and $J_+(u)+J_-(u)=2u$ for every $u\in (-1,1)$, we get 
\begin{align}\label{Hvalues}
\begin{split}
&\frac{\cj{x}_j}{x_j}\left(\frac{y_j-x_j}{2}\right)^2(H_+(t)+H_-(t))\\
& =\frac{(1+2Au+A^2)^2+(1-A^2)(1+2Au+A^2)-\frac{4(2A+(1+A^2)u)z}{y_j-x_j}}{(1+2Au+A^2)^2	\left(1+2u\left(A-\frac{2z}{y_j-x_j}\right)+\left(A-\frac{2z}{y_j-x_j}\right)^2\right)}.
\end{split}
\end{align}
Using the values that define $A$ and $u$ in \eqref{def-u-and-A}, we compute
\begin{align*}
1+2Au+A^2={} &\left(\frac{2}{y_j-x_j}\right)^2\frac{x_j}{\cj{x}_j}t,\\
2A+(1+A^2)u={} & \left(\frac{2}{y_j-x_j}\right)^2\frac{y_j^2+x_j^2}{y_j^2-x_j^2}\left(\frac{x_j}{\cj{x}_j}t-\frac{2y_j^2x_j^2}{y_j^2+x_j^2}\right),
\end{align*}
and
\begin{align*}
&1+2u\left(A-\frac{2z}{y_j-x_j}\right)+\left(A-\frac{2z}{y_j-x_j}\right)^2\\ &=
\frac{4}{(y_j-x_j)^2}(z-x_j)^2
+\frac{x_j}{\cj{x}_j}\frac{4\left(t-|x_j^2|\right)}{(y_j-x_j)^2}\left(1-\frac{2z}{y_j+x_j}\right).
\end{align*}
Assisted by these three identities, \eqref{Hvalues}  simplifies to
\begin{align*}
&H_+(|x_j|^2x)+H_-(|x_j|^2x) \\
& =-\frac{x^{-2}}{|x_j|^2}\frac{\left(\frac{y_j}{x_j}-1\right)(1-z/x_j)+(x-1)\left(\frac{y_j}{x_j}-2+z\frac{(y_j/x_j)^2+1}{y_j+x_j}\right)-(x-1)^2}{(1-z/x_j)^2+(x-1)\left(1-\frac{2z}{y_j+x_j}\right) }.
\end{align*}
Placing this expression into \eqref{Inexpress-1} we obtain
\begin{align}\label{I_n-formula-5-z}
\begin{split}
I_{n}(z)
={} &-\frac{|x_j|^{-2n-2}}{2\pi } \int_{1}^{\left|\frac{\lambda_j}{x_j}\right|^2} R_j(x-1,z)\frac{x^{-n-2}}{\sqrt{x-1}}dx+O(q_{j,r}^n),
\end{split}
\end{align}
uniformly in $z\in E$ as $n\to\infty$, with $R_j(w,z)$ given by \eqref{def-Rjk}.

When $r_j=|c_j|$, the computations are considerably  simpler. In this case,  we have
\[
 T_j(z)=\frac{x_j/\cj{x}_j}{2x_j-z},\qquad\frac{z}{T_j(z)}=|x_j|^2-\frac{\cj{x}_j}{x_j}\left(z-x_j\right)^2,
\]
so that $z/T_j(z)$ is a conformal map of any of the two half-planes that make up $\C\setminus C_j$ onto the exterior of $[|x_j|^2,\infty]$, the line $C_j$ being (doubly) mapped onto $[|x_j|^{2},\infty]$.  The inverse $g(t)$ of $z/T_j(z)$ is given by 
\[
 g(t)=x_j+\frac{x_j}{|x_j|}\sqrt{|x_j|^2-t},
\]
where the branch of the square root is the principal branch. Using these expressions, we compute 
\begin{align*}
\frac{T_j'(g(t))g'(t)}{g(t)-z}={} & -\frac{t^{-2}\left(1+\sqrt{1-\frac{t}{|x_j|^2}}\right)^2}{2\sqrt{1-\frac{t}{|x_j|^2}}\left(1-z/x_j+\sqrt{1-\frac{t}{|x_j|^2}}\right)},
\end{align*}
which combines with \eqref{I_n-formula-3-z} to yield \eqref{I_n-formula-5-z}, but  this time with $R_j(w,z)$ given by \eqref{def-Rjk-2}.

For every integer $m\geq 0$ we have 
\begin{align*}
 R_j(x-1,z)={} &\sum_{k=0}^m R_{j,k}(z)(x-1)^k +O\left((x-1)^{m+1}\right)
 \end{align*}
uniformly in $z\in E$ and $x\in [1, |\lambda_j/x_j|^2]$, which combined with \eqref{I_n-formula-5-z} produces
\begin{align}\label{asymptotic-formula-4}
\begin{split}
2\pi|x_j|^{2n+2}I_n(z)={}&-\sum_{k=0}^m R_{j,k}(z)\int_{1}^{\left|\frac{\lambda_j}{x_j}\right|^2 }x^{-n-2}(x-1)^{k-1/2}dx\\
&+O\left(\int_{1}^{\left|\frac{\lambda_j}{x_j}\right|^2}x^{-n-2}(x-1)^{m+1/2}dx\right)+O(|x_j|^{2n}q_{j,r}^n)
\end{split}
\end{align}
uniformly for $z\in E$ as $n\to\infty$. For $n\geq k\geq 0$, we have 
\begin{align}\label{asymptotic-formula-5}
\begin{split}
\int_{1}^{\left|\frac{\lambda_j}{x_j}\right|^2}x^{-n-2}(x-1)^{k-1/2}dx ={} &\int_{0}^{1}t^{n-k+1/2}(1-t)^{k-1/2}dt+O(|\lambda_j/x_j|^{-2n})\\
={} &\frac{\Gamma(n-k+3/2)\Gamma(k+1/2)}{\Gamma(n+2)}+O(|\lambda_j/x_j|^{-2n})\\
={} & \frac{\Gamma(k+1/2)}{(n+1)^{k+1/2}}(1+O(1/n))
\end{split}
\end{align}
as $n\to\infty$. The formula \eqref{asymptotic-formula-1-in-z} then follows from \eqref{asymptotic-formula-4} and \eqref{asymptotic-formula-5}.

 Let us finish with the observation that if $\Sigma_j$ denotes the component of the complement of $\T(0,|x_j|)\cup C_j$ with boundary $\partial \Sigma_j=\T(0,|x_j|)\cup C_j$, and we want to evaluate $I_n(z)$ for $z\not\in \cj{\Sigma}_j$, the first integral occurring in \eqref{I_n-formula-1-z} can be taken directly over $C_j$ (rather than over $L_{r,j}$), which leads by the same arguments used above to 
\begin{align}\label{direct-formula}
\begin{split}
I_{n}(z)
={} &-\frac{|x_j|^{-2n-2}}{2\pi } \int_{1}^{\left|\frac{y_j}{x_j}\right|^2} R_j(x-1,z)\frac{x^{-n-2}}{\sqrt{x-1}}dx,\qquad z\not\in \cj{\Sigma}_j. 
\end{split}
\end{align}
In particular, this applies to both $0$ and $p_j$, since these points  always fall outside $\cj{\Sigma}_j$. 

When $r_j=|c_j|$, we have $y_j=\infty$ and $R_j(w,0)\equiv -1$, so that  in  such a case \eqref{direct-formula} turns into \eqref{particular-case}.

 \subsection{Proof of Proposition \ref{prop-integral-error-2}}
 
 We will only  prove Proposition \ref{prop-integral-error-2} under the assumption that $r_j\not=|c_j|$, since the proof when $r_j=|c_j|$ is simpler and follows along similar lines. 

We begin with the identity
\begin{align}\label{I_n-formula-1}
\begin{split}
\chi_{n}(z):={} &  \frac{1}{2 \pi i}\ointctrclockwise\limits_{\T(0,|x_j|) } T'_j(\zeta)\frac{(T_j(\zeta)/\zeta)^{n}-(T_j(z)/z)^{n}}{\zeta - z}d\zeta\\
={} &
 \sum_{k=0}^{n-1}\frac{\left(T_j(z)/z\right)^k}{2 \pi i} \ointctrclockwise\limits_{\T(0,|x_j|) } 
\!\!\frac{\frac{T_j(\zeta)}{\zeta}-\frac{T_j(z)}{z}}{\zeta-z} T'_j(\zeta)\left(T_j(\zeta)/\zeta\right)^{n-1-k}d\zeta.
\end{split}
\end{align}
Using \eqref{formula-Tj} we find 
\begin{align*}
\frac{T_j(\zeta)-T_j(z)}{\zeta-z}={} & \frac{x_jy_j}{y_jx_j-\zeta\frac{y_j+x_j}{2}}\left(\frac{T_j(z)}{z}-\frac{1}{z}\frac{y_j+x_j}{2y_j\cj{x}_j}\right),
\end{align*}
and since 
\begin{align*}
\frac{\frac{T_j(\zeta)}{\zeta}-\frac{T_j(z)}{z}}{\zeta-z}=\frac{T_j(\zeta)-T_j(z)}{\zeta(\zeta-z)}-\frac{T_j(z)}{\zeta z},
\end{align*}
we can write \eqref{I_n-formula-1} in the form
\begin{align}\label{I_n-formula-2}
\begin{split}
\chi_{n}(z)
={} &-\sum_{k=0}^{n-1}\left(\frac{T_j(z)}{z}\right)^{k+1}I_{n-k-1}(p_j)\\
&-\frac{y_j+x_j}{2y_j\cj{x}_jz}\sum_{k=0}^{n-1}\left(\frac{T_j(z)}{z}\right)^k\left(I_{n-k-1}(0)-I_{n-k-1}(p_j)\right)
\end{split}
\end{align}
with $I_n(z)$ given by \eqref{def-In} (recall that $p_j$, given by \eqref{pj}, is the pole of $T_j$).

As noted at the end of the proof of Proposition \eqref{prop-integral-error-1}, we can use the formula \eqref{direct-formula} to evaluate $I_n(z)$ for $z=0$ and $z=p_j$.
Since
\begin{align*}
\begin{split}
R_j(x-1,0)=\frac{\epsilon_j \left(\frac{y_j}{x_j}-1\right)}{\sqrt{(y_j/x_j)^2-x}}-\frac{\epsilon_j \left(x-1\right)}{\sqrt{(y_j/x_j)^2-x}},
\end{split}
\end{align*}
and
\begin{align*}
\begin{split}
 R_j(x-1,p_j)
 ={} &-\frac{\epsilon_j \left(\frac{y_j}{x_j}+1\right)}{\sqrt{(y_j/x_j)^2-x}}+\left(\frac{y_j-x_j}{2}\right)^{-2}\frac{\epsilon_jy_jx_j\left(\frac{y^2_j}{x_j^2}-x\right)(x-1)}{x\sqrt{(y_j/x_j)^2-x}}\\
 &-\frac{\epsilon_j(x-1)}{\sqrt{(y_j/x_j)^2-x}},
\end{split}
\end{align*}
\eqref{direct-formula} gives us
\begin{align}\label{identity-1}
|x_j|^{2n+2}\left(I_{n}(0)-I_{n}(p_j)\right)
={} &  -\frac{\epsilon_jy_j}{\pi x_j} \mathcal{I}_{n,1} 
+\frac{\epsilon_jy_jx_j}{2\pi}\left(\frac{2}{y_j-x_j}\right)^{2}\mathcal{I}_{n,2},
\end{align}
and
\begin{align}\label{identity-2}
\begin{split}
|x_j|^{2n+2}I_n(p_j)= {} &\frac{\epsilon_j\left(\frac{y_j}{x_j}+1\right)}{2\pi}\mathcal{I}_{n,1}-\frac{\epsilon_jy_jx_j}{2\pi\left(\frac{y_j-x_j}{2}\right)^2}\mathcal{I}_{n,2} +\frac{\epsilon_j}{2\pi}\mathcal{I}_{n,3},
  \end{split}
\end{align}
where 
\begin{align*}
\mathcal{I}_{n,1}:={} & \int_{1}^{\left|\frac{y_j}{x_j}\right|^2}\frac{x^{-n-2}dx}{\sqrt{(y_j/x_j)^2-x}\sqrt{x-1}},\\
\mathcal{I}_{n,2}:={} & \int_{1}^{\left|\frac{y_j}{x_j}\right|^2}\sqrt{(y_j/x_j)^2-x}\sqrt{x-1}x^{-n-3} dx,\\
\mathcal{I}_{n,3}:={} &\int_{1}^{\left|\frac{y_j}{x_j}\right|^2}\frac{\sqrt{x-1}x^{-n-2}dx}{\sqrt{(y_j/x_j)^2-x}}.
\end{align*}

Setting $t=|x_j|^{2}\frac{T_j(z)}{z}$, we can use \eqref{identity-1} and \eqref{identity-2} to write \eqref{I_n-formula-2} in the form
\begin{align}\label{I_n-dominant-term}
\begin{split}
|x_j|^{2n}\chi_{n}(z)={} &-
\frac{\epsilon_j\left(\frac{y_j}{x_j}+1\right)}{2\pi}\left(\frac{T_j(z)}{z}-\frac{1}{\cj{x}_jz}\right)\sum_{k=0}^{n-1}t^k\mathcal{I}_{n-k-1,1}\\
&+\frac{\epsilon_j}{\pi}\frac{2y_jx_j}{(y_j-x_j)^2}\left(\frac{T_j(z)}{z}-\frac{1}{\cj{x}_jz}\right)\sum_{k=0}^{n-1}t^k\mathcal{I}_{n-k-1,2}\\
&+\frac{\epsilon_jx_j}{\pi\cj{x}_j(y_j-x_j)z}\sum_{k=0}^{n-1}t^k\mathcal{I}_{n-k-1,2}\\
&- \frac{\epsilon_j}{2\pi}\frac{T_j(z)}{z}\sum_{k=0}^{n-1}t^k\mathcal{I}_{n-k-1,3}.
\end{split}
\end{align}

 To prove \eqref{BigO-formula-1iminf}, we make $z=x_j$ (recall that $T_j(x_j)=\cj{x}_j^{-1}$) in  \eqref{I_n-dominant-term}, which reduces to   
\begin{align*}
\begin{split}
|x_j|^{2n}\chi_{n}(x_j)={} 
&\frac{\epsilon_j|x_j|^{-2}}{\pi(y_j/x_j-1)}\int_{1}^{\left|\frac{y_j}{x_j}\right|^2}\frac{\sqrt{(y_j/x_j)^2-x}(1-x^{-n})}{x^2\sqrt{x-1}} dx\\
&- \frac{\epsilon_j|x_j|^{-2}}{2\pi}\int_{1}^{\left|\frac{y_j}{x_j}\right|^2}\frac{(1-x^{-n})dx}{x\sqrt{(y_j/x_j)^2-x}\sqrt{x-1}}.
\end{split}
\end{align*}
Letting $n\to\infty$ and computing the resulting integrals yields  
\begin{align*}
\begin{split}
\lim_{n\to\infty}|x_j|^{2n+2}\chi_{n}(x_j)
={}&\frac{\epsilon_j}{\pi(y_j/x_j-1)}\frac{\pi(|y_j/x_j|^2-1)}{2|y_j/x_j|}- \frac{\epsilon_j}{2\pi}\frac{\pi}{|y_j/x_j|}\\
={}&-\frac{1}{2}.
\end{split}
\end{align*}

We now prove \eqref{BigO-formula-1}. Just as we deduced \eqref{asymptotic-formula-1-in-z} from \eqref{I_n-formula-5-z}, we  can verify that, as $n\to\infty$, 
$\mathcal{I}_{n,1}=O(n^{-1/2})$, while $\mathcal{I}_{n,2}$ and $\mathcal{I}_{n,3}$ are both $O(n^{-3/2})$. These estimates can be used in \eqref{I_n-dominant-term} to deduce that
\begin{align*}
|x_j|^{2n}|\chi_n(z)|
={} &O\left(\left|\frac{T_j(z)}{z}-\frac{1}{z \cj{x}_j}\right|\sum_{k=0}^{n-1}\frac{|t|^k}{\sqrt{n-k}}\right)+ O\left(\sum_{k=0}^{n-1}\frac{|t|^k}{(n-k)^{3/2}}\right).
\end{align*}

With the aid of \eqref{formula-Tj}, we find
\begin{align*}
\frac{T_j(z)}{z}-\frac{1}{z\cj{x}_j} =\frac{z-x_j}{z\cj{x}_j}\frac{\frac{y_j-x_j}{2}}{y_jx_j-z\frac{y_j+x_j}{2}},
\end{align*}
\begin{align*}
\frac{T_j(z)}{z}-\frac{1}{|x_j|^2} =\frac{(z-x_j)^2}{z|x_j|^2}\frac{\frac{y_j+x_j}{2}}{y_jx_j-z\frac{y_j+x_j}{2}},
\end{align*}
so that 
\begin{align*}
|z||x_j|^2\frac{2}{|y_j+x_j|}\left|y_jx_j-z\frac{y_j+x_j}{2}\right|\left|\frac{T_j(z)}{z}-\frac{1}{|x_j|^2}\right| =|z-x_j|^2,
\end{align*}
and combining these three relations we obtain
\begin{align*}
\left|\frac{T_j(z)}{z}-\frac{1}{z\cj{x}_j}\right|=\frac{\frac{|y_j-x_j|}{\sqrt{2|z||y_j+x_j|}|x_j|^2}\sqrt{\left|\frac{T_j(z)}{z}-\frac{1}{|x_j|^2}\right|}}{\sqrt{\left|y_jx_j-z\frac{y_j+x_j}{2}\right|}} .
\end{align*}
Thus,
\begin{align*}
|x_j|^{2n}|\chi_n(z)|
={} &O\left(\sum_{k=0}^{n-1}\frac{|t|^k\sqrt{\left|t-1\right|}}{\sqrt{n-k}}\right) + O\left(\sum_{k=0}^{n-1}\frac{|t|^k}{(n-k)^{3/2}}\right).
\end{align*}

We now observe that $|t|\leq 1$ for $|z|=|x_j|$ because the mapping $t=|x_j|^2T_j(z)/z$ takes the circle $|z|=|x_j|$ onto a closed Jordan curve, all points of which are contained on $|w|<1$ except for the point $t=1$. This Jordan curve is symmetric about the $x$-axis with a  cusp at $t=1$ (forming a $0$ angle with the $x$-axis). It then follows from elementary geometric arguments that there is a constant $M_j'$ such that 
\[
 \frac{|t-1|}{1-|t|}\leq M_j',\quad t=|x_j|^2T_j(z)/z,\quad |z|=|x_j|.
\]
Hence, 
\begin{align*}
|x_j|^{2n}|\chi_n(z)|
\leq O\left(\sum_{k=0}^{n-1}\frac{|t|^k\sqrt{1-|t|}}{\sqrt{n-k}}\right)+O(1)
\end{align*}
uniformly on $|z|=|x_j|$ as $n\to\infty$.

So to finish the proof of \eqref{BigO-formula-1} we only need to bound the sum on the right-hand side of the previous inequality. Let $q\geq0$. By finding the extreme values of the  function $\tau^{q}\sqrt{1-\tau}$ we can easily see that $\tau^{q}\sqrt{1-\tau}< (q+1)^{-1/2}$  for every  $\tau\in [0,1]$, so that
\begin{align*}
\sum_{k=0}^{n-1}\frac{\tau^{k}\sqrt{1-\tau}}{\sqrt{n-k}}\leq {} &\sum_{k=0}^n\frac{\frac{1}{\sqrt{n-k}}+\frac{1}{\sqrt{k+1}}}{\sqrt{n-k}+\sqrt{k+1}}
\leq \frac{2}{\sqrt{n}+1}\sum_{k=1}^{n}\frac{1}{\sqrt{k}}\\
\leq {} &\frac{2}{\sqrt{n}+1}\left(1+\int_1^{n}x^{-1/2}\right)\leq 4.
\end{align*}
This finishes the proof of \eqref{BigO-formula-1}.

\subsection{Proof of Proposition \ref{cdcd}}
The asymptotic behavior of a series very much like \eqref{series-behavior} has already been established in \cite{DMN}, but the technical details of the proof were given in the expanded version \cite{DMN-ARXIV}. Our job here is more of translating what was accomplished there into our current setting. We will therefore lay out the main steps involved in proving \eqref{series-behavior}, indicating in each case where to find the full explanation in \cite{DMN-ARXIV}.

Let $B_j$ be the disk defined by \eqref{disk-Bj}, and note that $\Phi_{j}(D(0,|a_j|))= B_j$. Since we want to analyze the left-hand side of \eqref{series-behavior} for $z\in D(0,|a_j|)$, and since $T_j^v(z)=\Phi_j^{-1}(\sigma^{2v}\Phi_j(z))$, it is equivalent to make the change of variable $z=\Phi_j^{-1}(t)$ and analyze the simpler expression that results for $t\in B_j$.  
This change of variables gives 
\begin{align}\label{after-change-variables}
\left.\sum_{v=1}^\infty   (T_j ^v(z))^n (T_j ^v)' (z) \right|_{z=\Phi_j^{-1}(t)}= \frac{a_j^n}{\lambda'_j(t)} \sum_{v=1}^\infty  \sigma_j^{2v} G_{j,n}(\sigma_j^{2v} t),
\end{align}
where 
\begin{align}\label{definition-G-functions}
\lambda_j(t)=\frac{t+|a_j|}{1+|a_j|t},\quad G_{j,n} (t)= \lambda_j '(t)\left(
\frac{\lambda_j(t)}
{|a_j|} \right)^n.
\end{align}
Observe that  $\Phi_j^{-1}(t)=\frac{a_j}{|a_j|}\lambda_j(t)$ and so $\lambda_j(B_j)=D(0,|a_j|)$.

Since 
\[
\left(\frac{\lambda_j(t/n)}
{|a_j|} \right)^n=\left(1+\frac{\alpha_jt/n}{1+|a_j|t/n}\right)^n,
\] 
this quantity should converge to an exponential as $n\to\infty$. Indeed, by following the arguments in the proof of \cite[Lemma 3.1]{DMN-ARXIV}, we can prove that for every compact set $E \subset B_j$, there exist positive constants $m$ and $M$ such that for every integer $n \geq 1$, we have
\begin{align}\label{estimate-exponential}
\left| 
e^{\alpha_jtx} - |a_j|^{-n} \lambda_j^n(tx/n) 
\right|
\leq 
\frac{Mx^2e^{-mx}}{n},
\quad
t\in E,
\quad
0 \leq x \leq n.
\end{align}

We now seek to find integral representations for the series in the right-hand side of \eqref{after-change-variables} and for $\Theta_{\sigma_j^2}(t)$ as defined by \eqref{definition-theta}. This can be accomplished via the functions 
\[
S_j(x) := \sigma_j^{2 \lfloor \log_{\sigma_j^2} (x)  \rfloor}=x \sigma_j^{-2   \langle  \log_{\sigma_j^2} (x)  \rangle }
, \quad x \in (0, \infty),\]
where   $\lfloor x  \rfloor$ and $\langle x \rangle$ denote the integer and fractional part of $x$, respectively. Note that for every $x>0$, we have
\begin{equation*}
x \in (\sigma_j^{2(v+1)}, \sigma_j ^{2v}] 
\quad
\Leftrightarrow
\quad
S_j(x) = \sigma_j^{2v}
 \end{equation*}
 and 
 \begin{align}\label{bounds}
  x\leq S_j(x)\leq \sigma_j^{-2} x,\qquad x>0.
 \end{align}
Note also that 
\begin{align}\label{integral-repres-1}
\int_0 ^1 S_j(x)  f(x)   dx &= 
\sum_{v =0}  ^\infty \int_{\sigma_j^{2v+2}} ^{\sigma_j^{2v}} \sigma_j^{2v}  f(x)   dx,
\end{align}
\begin{align*}\
\int_0 ^\infty S_{j} (x)  f(x)   dx &=
\sum_{v \in \Z} \int_{\sigma_j^{2v+2}} ^{\sigma_j^{2v}} \sigma_j^{2v} f(x)   dx .
\end{align*}

Using the latter identity, we find the first integral representation
\begin{equation}\label{Theta}
\Theta_{\sigma_j^2}(t) = -\frac{ \sigma_j^2 t^2 }{1-\sigma_j^2 }    \int_{0} ^\infty S_{j} (x) e^{t x} dx,\qquad \Re{t}<0.
\end{equation}
Indeed, 
\begin{align*}
t \int_{0} ^\infty S_{j} (x) e^{  t x} dx 
={} & \sum_{v \in \Z} \sigma_j^{2v}\int_{\sigma_j^{2v+2}} ^{\sigma_j^{2v}}   e^{t x} t dx=\sum_{v \in \Z}  \sigma_j^{2v} 
e^{ \sigma_j^{2v}t}  
-  \sum_{v \in \Z}  \sigma_j^{2v}  e^{  \sigma_j^{2v+2}t}
\\
={} & 
\frac{(\sigma_j^2-1)\Theta_{\sigma_j^2}(t)}{\sigma_j^2t}.
\end{align*}

A consequence of \eqref{Theta} (see formula (33) in \cite{DMN-ARXIV} and its subsequent derivation) is that 
\begin{align}\label{diference-theta}
\Theta_{\sigma_j^2}(nt) - \Theta_{\sigma_j^2}((n+1)t) = O(1/n)
\end{align}
uniformly as $n\to\infty$ on compact subsets of $\Re{t}<0$.
 
The second integral representation is
 \begin{align}\label{integral-represetation-3}
 \sum_{v=0}^\infty \sigma_j^{2v} G_{j,m}(\sigma_j^{2v} t) = \frac{G_{j,m}(t)}{1-\sigma_j^2} -  \frac{\sigma_j^2t }{(1-\sigma_j^2)n }  \int_0 ^{n}  S_j(x/n) G_{j,m}'(xt/n)  dx,
 \end{align}
which is valid for every pair $(m,n) \in \N \times \N$ and for every  $t \in \C\setminus (-\infty,-|a_j|^{-1}]$. This representation can be derived by  
using the summation by parts formula in conjunction with \eqref{integral-repres-1}; for the details see Lemma 3.3 in \cite{DMN-ARXIV} and its proof.

From the definition \eqref{definition-G-functions}, we find  
\begin{align}\label{Gprime}
G'_{j,n+1}(t) = \frac{n+1}{|a_j|}  \left(\frac{\lambda_j(t)}{|a_j|}\right)^{n} \mathcal{L}_{j,n}(t),
\end{align}
with  
\begin{equation*}
\mathcal{L}_{j,n} (t) =  \frac{|a_j|^2\alpha_j^2}{(1+|a_j|t)^4} - \frac{2|a_j|^2\alpha_j(t+|a_j|)}{(n+1)(1+ |a_j|t)^4},
\end{equation*}
so that, as $n \to \infty$,
\begin{align}\label{Ljn}
\mathcal{L}_{j,n} (xt/n) 
&=|a_j|^2\alpha_j^2 +O(x/n)+O(1/n)
\end{align}
uniformly for $t$ on compact subsets of $\C\setminus (-\infty,-|a_j|^{-1}]$ and $x\in  [0,n]$.

It follows from \eqref{Theta}, \eqref{bounds} and \eqref{Ljn} that
\begin{align}\label{thetant}
\begin{split}
\Theta_{\sigma_j^2}(n\alpha_jt) ={} & -\frac{ \sigma_j^2 \alpha_j^2t^2 }{1-\sigma_j^2 }    \int_{0} ^\infty nS_{j} (x/n) e^{\alpha_jt x} dx\\
={} & -\frac{ \sigma_j^2|a_j|^{-2}t^2 }{1-\sigma_j^2 }   \int_{0} ^n nS_{j} (x/n)e^{\alpha_j t x}  \mathcal{L}_{j,n} (xt/n)  dx  + O(1/n)
\end{split}
\end{align}
locally uniformly on $\C\setminus (-\infty,-|a_j|^{-1}]$ as $n \to \infty$. 

From the relations \eqref{Gprime}, \eqref{Ljn}, \eqref{thetant}, and \eqref{estimate-exponential} we  see that for every compact set $E\subset B_j$, there exist positive constants $m$ and $M'$ such that 
\begin{align*}
&\left|
 \int_{0} ^n nS_j(x/n)    G'_{j,n+1} (xt/n)  dx
+
\frac{(n+1)|a_j|(1-\sigma_j^2)}{ \sigma_j^2 t^2 }\Theta_{\sigma_j^2}(n\alpha_jt)
\right|
\\
&\leq \frac{n+1}{|a_j|}
\int_0 ^n nS_{j} (x/n) \left| \frac{\lambda_j(tx/n)^n}{|a_j|^n} -e^{\alpha_jtx} \right|  | \mathcal{L}_{j,n}\left(\frac{tx}{n}\right) |  dx + O(1)
\\
& \leq M' \int_{0} ^\infty x^3 e^{-mx}  dx + O(1)
\end{align*} 
uniformly in $t \in E$ as $n \to \infty$. 

Since $|\lambda_j(t)|<|a_j|$ for $t\in B_j$, this latter estimate plus  \eqref{integral-represetation-3} and \eqref{diference-theta} yield
\begin{align}\label{before-substitution}
\sum_{v=0}^\infty  \sigma_j^{2v} G_{j,n}(\sigma_j^{2v} t)=\frac{|a_j|\Theta_{\sigma_j^2}(n\alpha_jt) }{  nt }+O(n^{-2})
\end{align}
locally uniformly on  $B_j$ as $n \to \infty$. 

We now observe that if $E$ is a closed subset of $\cj{B}_j\setminus\{0\}$, then $\sigma_j^2E=\{\sigma_j^2t:t\in E\}$ is a closed subset of $B_j$, and as noted before, $\Theta_{\sigma_j^2}(\sigma_j^2t)=\Theta_{\sigma_j^2}(t)$, so that by replacing $t$ by $\sigma_j^2t$ in  \eqref{before-substitution} and using the resulting formula in \eqref{after-change-variables}, we conclude that 
\begin{align}\label{before-substitution-1}
\left.\sum_{v=1}^\infty   (T_j ^v(z))^n (T_j ^v)' (z) \right|_{z=\Phi_j^{-1}(t)}=\frac{a_j^n}{\lambda'_j(t)}\frac{|a_j|\Theta_{\sigma_j^2}(n\alpha_jt)}{  nt }+O(|a_j|^n/n^2)
\end{align}
uniformly on closed subsets of $\cj{B}_j\setminus\{0\}$ as $n \to \infty$. Since $\Phi_j$ takes $\cj{D(0,|a_j|)}\setminus\{a_j\}$ onto $\cj{B}_j\setminus\{0\}$, 
\eqref{series-behavior} follows by making $t=\Phi_j(z)$ in \eqref{before-substitution-1}.

\end{document}